\documentclass[reqno,10pt, centertags]{amsart}
\usepackage{amsmath,amsthm,amscd,amssymb,latexsym,esint,upref,stmaryrd,
enumerate,color,verbatim,yfonts}
\usepackage{hyperref} 
\newcommand*{\mailto}[1]{\href{mailto:#1}{\nolinkurl{#1}}}
\newcommand{\arxiv}[1]{\href{http://arxiv.org/abs/#1}{arXiv:#1}}

\newcommand{\bbC}{{\mathbb{C}}}

\newcommand{\bbN}{{\mathbb{N}}}

\newcommand{\bbR}{{\mathbb{R}}}

\newcommand{\cB}{{\mathcal B}}

\newcommand{\cD}{{\mathcal D}}

\newcommand{\cH}{{\mathcal H}}

\newcommand{\cX}{{\mathcal X}}

\newcommand{\beq}{\begin{equation}}
\newcommand{\enq}{\end{equation}}

\renewcommand{\a}{\alpha}
\renewcommand{\b}{\beta}
\newcommand{\g}{\gamma}

\newcommand{\z}{\zeta}

\DeclareMathOperator{\supp}{supp}

\DeclareMathOperator{\dom}{dom}

\renewcommand{\Re}{\text{\rm Re}}
\renewcommand{\Im}{\text{\rm Im}}
\renewcommand{\ln}{\text{\rm ln}}

\newcommand{\no}{\notag}
\newcommand{\lb}{\label}
\newcommand{\f}{\frac}

\newcommand{\ol}{\overline}

\newcommand{\wti}{\widetilde}
\newcommand{\Oh}{O}

\newcommand{\hatt}{\widehat} 
\newcommand{\dott}{\,\cdot\,}

\newcommand{\bi}{\bibitem}

\renewcommand{\ge}{\geqslant}

\makeatletter

\allowdisplaybreaks 
\numberwithin{equation}{section}

\newtheorem{theorem}{Theorem}[section]

\newtheorem{lemma}[theorem]{Lemma}
\newtheorem{corollary}[theorem]{Corollary}

\newtheorem{example}[theorem]{Example}

\theoremstyle{remark}
\newtheorem{remark}[theorem]{Remark}

\begin{document}
\title[A Survey of Some Norm Inequalities]{A Survey of Some Norm Inequalities} 

\author[F.\ Gesztesy]{Fritz Gesztesy}
\address{Department of Mathematics, 
Baylor University, Sid Richardson Bldg., 1410 S.\,4th Street, Waco, TX 76706, USA}
\email{\mailto{Fritz\_Gesztesy@baylor.edu}}
\urladdr{\url{http://www.baylor.edu/math/index.php?id=935340}}

\author[R.\ Nichols]{Roger Nichols}
\address{Department of Mathematics (Dept.~6956), The University of Tennessee at Chattanooga, 
615 McCallie Ave, Chattanooga, TN 37403, USA}
\email{\mailto{Roger-Nichols@utc.edu}}
\urladdr{\url{http://www.utc.edu/faculty/roger-nichols/index.php}}

\author[J.\ Stanfill]{Jonathan Stanfill}
\address{Department of Mathematics, 
Baylor University, Sid Richardson Bldg., 1410 S.\,4th Street, Waco, TX 76706, USA}
\email{\mailto{Jonathan\_Stanfill@baylor.edu}}
\urladdr{\url{http://sites.baylor.edu/jonathan-stanfill/}}

\dedicatory{Dedicated with great pleasure to Henk de Snoo on the happy occasion of his 75th birthday}

\date{\today}

\thanks{Originally appeared in {\it Complex Analysis and Operator Theory} {\bf 15}, No.~23 (2021); the present version contains some updates.}
\@namedef{subjclassname@2020}{\textup{2020} Mathematics Subject Classification}

\begin{abstract}
We survey some classical norm inequalities of Hardy, Kallman, Kato, Kolmogorov, Landau, Littlewood, and Rota of the type 
\[
\|A f\|_{\cX}^2 \leq C \|f\|_{\cX} \big\|A^2 f\big\|_{\cX}, \quad f \in \dom\big(A^2\big), 
\]
and recall that under exceedingly stronger hypotheses on the operator $A$ and/or the Banach space $\cX$, the optimal constant $C$ in these inequalities diminishes from $4$ (e.g., when $A$ is the generator of a $C_0$ contraction semigroup on a Banach space $\cX$) all the way down to $1$ (e.g., when $A$ is a symmetric operator on a Hilbert space $\cH$).  

We also survey some results in connection with an extension of the Hardy--Littlewood inequality involving quadratic forms as initiated by Everitt. 
\end{abstract}

\maketitle 

\section{Introduction} \lb{s1} 

{\it We dedicate this note with great pleasure to Henk de Snoo, whose exemplary scholarship over the years  deserves our undivided respect and admiration. Happy Birthday, Henk, we hope our modest contribution to norm inequalities will give some joy.} 

This is a survey of a number of classical norm inequalities due to Hardy, Kallman, Kato, Kolmogorov, Landau, Littlewood, and Rota of the type 
\begin{equation} 
\|A f\|_{\cX}^2 \leq C \|f\|_{\cX} \big\|A^2 f\big\|_{\cX}, \quad f \in \dom\big(A^2\big),    \lb{1.1}
\end{equation} 
and some of their higher-order extensions. In particular, we recall that under exceedingly stronger hypotheses on the operator $A$ and/or the Banach space $\cX$, such as, \\[1mm] 
$(i)$ \;\,\,$A$ is the generator of a $C_0$ contraction semigroup on a Banach space $\cX$, \\[1mm]  
$(ii)$\;\,\,$A$ is the generator of a $C_0$ contraction semigroup on a Hilbert space $\cH$, \\[1mm]  
$(iii)$\;$A$ is the generator of a $C_0$ group of isometries on a Banach space $\cX$, \\[1mm] 
$(iv)$ $A = c S^*$, where $c \in \bbC$ and $S$ is maximally symmetric in a Hilbert space but 
\hspace*{5.5mm}  not self-adjoint, \\[1mm] 
$(v)$\;\,\,$A$ is a symmetric operator on a Hilbert space $\cH$, \\[1mm]  
the optimal constant $C$ in these inequalities diminishes from $4$ to $1$ in the process.

Historically, this type of investigations started with Landau \cite{La13} in 1913 who proved 
\begin{equation}
\|f'\|_{L^{\infty}((0,\infty);dx)}^2 \leq 4 \, \|f\|_{L^{\infty}((0,\infty);dx)} \|f''\|_{L^{\infty}((0,\infty);dx)},  \quad 
f\in W^{2,\infty}((0,\infty))     \lb{1.2} 
\end{equation}
(the constant $4$ being optimal), followed by Hardy and Littlewood \cite{HL32} who derived the $L^2$-analog of \eqref{1.2} in 1932, 
\begin{align}
\|f'\|_{L^2((0,\infty);dx)}^2 \leq 2 \, \|f\|_{L^2((0,\infty);dx)} \|f''\|_{L^2((0,\infty);dx)},  \quad 
f \in W^{2,2}((0,\infty))    \lb{1.3} 
\end{align}
(again, with best possible constant $2$). These authors also proved the analogs of inequalities \eqref{1.2} and \eqref{1.3} on the whole line $\bbR$, with (optimal) constants $2$ (see also Hadamard \cite{Ha14} in this context) and $1$, respectively, followed by fundamental work of Kolmogorov \cite{Ko39} in 1939. These early investigations led to extensive subsequent work in this area as will be shown in the bulk of this survey. 

Inequalities of the type \eqref{1.2}, \eqref{1.3} were abstracted in the form \eqref{1.1} with $C=4$ by Kallman and Rota \cite{KR70} in the context where $A$ is the generator of a $C_0$ contraction semigroup on a Banach space $\cX$. That the constant can be diminished from $4$ to $2$ in the Hilbert space context was shown by Kato \cite{Ka71} in 1971. Again, this marked the beginning of numerous subsequent works, especially in connection with higher-order analogs of the estimate \eqref{1.1}. 

In Section \ref{s2} we survey the case of norm inequalities for generators of $C_0$ semigroups in Banach and Hilbert spaces. The case where $A$ generates a $C_0$ group in Banach and Hilbert spaces is recalled in Section \ref{s3}. Some inequalities for fractional powers of generators of contraction semigroups are surveyed in Section \ref{s4}. Extensions of the Hardy--Littlewood inequality in Hilbert spaces are recalled in Section \ref{s5}. In particular, we discuss an extension initiated by Everitt \cite{Ev71} in 1971 involving quadratic forms of general Sturm--Liouville operators and then add some considerations naturally involving the Friedrichs extension $A_F$ of a symmetric operator $A$ bounded from below. The explicitly solvable example associated with the differential expression
\begin{align}
\begin{split}
\tau_{\a,\b,\g} = x^{-\a}\left[-\frac{d}{dx}x^\b\frac{d}{dx} +\frac{(2+\a-\b)^2\g^2-(1-\b)^2}{4}x^{\b-2}\right],\\
\a>-1,\ \b<1,\ \g \in (0,1),\ x\in(0,\infty),     \lb{1.4} 
\end{split}
\end{align}
is analyzed in some detail in our final Section \ref{s6}.

For other surveys of many aspects of integral inequalities we refer, for instance, to \cite{BBBBDEEKL98}, 
\cite{EE82}, \cite{EEHJ98}, \cite{Ev75}, \cite{KZ80}, \cite{KZ81}, \cite{KZ91}, and \cite{KZ92}.

Finally, some comments regarding our notation: All Hilbert spaces $\cH$ are assumed to be complex in this survey and a symmetric operator $A$ in $\cH$ is always assumed to be densely defined 
(such that $A \subseteq A^*$).

\section{Norm Inequalities for Generators of $C_0$ Semigroups} \lb{s2} 

We begin by considering inequalities concerning (infinitesimal) generators $G$ of $C_0$ semigroups of bounded operators in the form $T(t)=e^{t G}$, $t\in[0,\infty)$, on a Banach space $\cX$, in particular, $\{T(t)\}_{t \in [0,\infty)} \subset \cB(\cX)$ satisfies \\[1mm]
$(i)$ $T(0) = I_{\cX}$; \\[1mm]
$(ii)$ $T(s) T(t) = T(s+t)$, $s, t \in [0,\infty)$; \\[1mm]
$(iii)$ $[0,\infty) \ni t \mapsto T(t)f \in \cX$ is continuous for each $f \in \cX$ (w.r.t. the topology on $\cX$, i.e., 
$T(\dott)f \in C([0,\infty), \cX)$). 

We will especially be interested in the case of $C_0$ contraction semigroups, that is, those satisfying 
$\|T(t)\|_{\cB(\cX)} \leq 1$, $t \in [0,\infty)$. 

Given a $C_0$ semigroup $T(t)$, $t \in [0,\infty)$, its generator $G$ is then defined as usual via
\begin{equation}
Gf = \lim_{t \downarrow 0} t^{-1} [T(t) f - f],    \lb{2.1} 
\end{equation}
with $\dom(G) \subseteq \cX$ consisting precisely of those $f \in \cX$ for which the limit in \eqref{2.1} exists in the norm $\| \dott \|_{\cX}$ of $\cX$. 

\begin{theorem}\lb{t2.1}  $($The Kallman--Rota inequality \cite{KR70}, see also \cite[Theorem 9.8]{Go85}$)$.
Let $c \in \bbC\backslash\{0\}$ and suppose that $cA$ generates a $C_0$ contraction semigroup $T(t)$,  $t \in [0,\infty)$, on a Banach space $\cX$. Then  
\begin{equation}
\|Af\|_{\cX}^2\leq 4 \, \|f\|_{\cX} \big\|A^2f\big\|_{\cX}, \quad f\in\dom\big(A^2\big).     \lb{2.2} 
\end{equation}
\end{theorem}
\begin{proof}[Sketch of proof.] There are at least two well-known proofs of Theorem \ref{t2.1}. Without loss of generality we assume that $c=1$ in the following. \\[1mm] 
$(i)$ The standard semigroup proof, presented, for instance, in \cite[Theorem 9.8]{Go85}, repeatedly uses the fact $(d/ds)T(s)f=AT(s)f=T(s)Af$, $f\in\dom(A)$, and then derives the relation 
\begin{equation}\lb{2.8}
tAf=T(t)f-f-\int_0^tds\ (t-s)T(s)A^2f, \quad f\in\dom\big(A^2\big). 
\end{equation}
The triangle inequality and the contraction property of $T(\, \cdot \,)$ imply 
\begin{align}\lb{2.9}
\no  \|Af\|_{\cX} & \leq t^{-1}[\|T(t)f\|_{\cX} + \|f\|_{\cX}] + t^{-1}\int_0^tds\ (t-s)\big\|T(s)A^2f\big\|_{\cX} \\
\no &\leq t^{-1}[\|f\|_{\cX} + \|f\|_{\cX}] + t^{-1}\int_0^tds\ (t-s)\big\|A^2f\big\|_{\cX} \\
&=\dfrac{2}{t} \, \|f\|_{\cX} + \dfrac{t}{2} \, \big\|A^2f\big\|_{\cX},\quad f\in\dom\big(A^2\big).
\end{align}
If $A^2f=0$, letting $t\uparrow\infty$ in \eqref{2.9} shows that $Af=0$. If $A^2f\neq0$, one minimizes the right-hand side of \eqref{2.9} over $t > 0$ by noting that the minimum occurs at $t=2 \, \|f\|_{\cX}^{1/2}\big/\big\|A^2f\big\|_{\cX}^{1/2}$, proving \eqref{2.2}.  \\[1mm] 
$(ii)$ A functional analytic proof of \eqref{2.2} was presented by Certain and Kurtz \cite{CK77}. It uses Landau's inequality \cite{La13}
\begin{align}
&\|f'\|_{L^{\infty}((0,\infty);dx)}^2 \leq 4 \, \|f\|_{L^{\infty}((0,\infty);dx)} \|f''\|_{L^{\infty}((0,\infty);dx)},  \no \\
& f\in \big\{g \in L^{\infty}((0,\infty);dx) \, \big| \, g,g'\in AC([0,R])\text{ for all }R>0,  
\lb{2.9a} \\
& \hspace{5.7cm} g''\in L^{\infty}((0,\infty);dx)\big\}     \no \\
& \hspace*{6mm} \quad = W^{2,\infty}((0,\infty)),     \no 
\end{align}
as input and then derives the following inequality from it: Let $\cX$ be a Banach space with norm 
$\|\, \cdot \,\|_{\cX}$, and functions $F \colon [0,\infty) \to \cX$ with $||| \, \cdot \, |||$ abbreviating 
\begin{equation}
|||F||| = \sup_{t \in [0,\infty)} \|F(t)\|_{\cX}. 
\end{equation}
Next, assume that $F \colon [0,\infty) \to \cX$ has two continuous derivatives with $|||F||| < \infty$ and 
$|||F''||| < \infty$. If the functional $\ell \in \cX^*$ is normalized, $\|\ell\|_{\cX^*} =1$, then the function 
$g(t) = \ell(F(t))$, $t \in [0,\infty)$, is twice continuously differentiable with $g^{(k)}(t) = \ell\big(F^{(k)}(t)\big)$, $1 \leq k \leq 2$. Hence Landau's inequality \eqref{2.9a} for $g$ yields
\begin{align}
\big(\sup_{t \in [0,\infty)} |\ell(F'(t))|\big)^2 
& \leq 4 \sup_{t \in [0,\infty)} |\ell(F(t))| \sup_{t \in [0,\infty)} |\ell(F''(t))|   \no \\
& \leq 4 \, |||F||| \,  |||F''|||
\end{align} 
since $\|\ell\|_{\cX^*} =1$. Taking the supremum over all $\ell \in \cX^*$ with $\|\ell\|_{\cX^*} =1$, employing 
the fact that $\|F(t)\|_{\cX} = \sup_{\ell \in \cX^*, \|\ell\|_{\cX^*}=1} |\ell(F(t))|$, $t \in [0,\infty)$, proves
\begin{equation}
|||F'|||^2 \leq 4 \, |||F||| \, |||F''|||.   \lb{2.9b}
\end{equation}
If $A$ is the generator of a $C_0$ contraction semigroup $T(t)$, $t \in [0,\infty)$, then for $f \in \dom\big(A^2\big)$, $F(t) = T(t) f$ is twice continuously differentiable and combining \eqref{2.9b} with 
\begin{equation}
\big|\big|\big| F^{(k)} \big|\big|\big| = \sup_{t \in [0,\infty)} \big\| T(t) A^k f\big\|_{\cX} = \big\|A^k f\big\|_{\cX}, 
\quad k = 0, 1, 2, 
\end{equation}
yields the estimate \eqref{2.2}. 
\end{proof}

\begin{remark} \lb{r2.2}
$(i)$ For additional literature in the context of Hardy, Kallman, Landau, Littlewood, Rota inequalities we refer to 
\cite{Bo73}, \cite{Bo73a}, \cite{Co77}, \cite{EJ77}, \cite{EZ78}, \cite{GG75}, \cite{GKZ83}, \cite{Ho74}, \cite{KZ79}, \cite{KZ79a}, \cite{KZ80}, \cite{KZ81}, \cite{KZ81a}, \cite{KZ84}, \cite[Ch.~2]{KZ92}, \cite{Pa81}.  \\[1mm] 
$(ii)$ The identical strategy of proof in \cite{CK77} extends the higher-order $L^{\infty}$-inequalities,
\begin{align} 
& \big\|f^{(k)}\big\|_{L^{\infty}((0,\infty);dx)}^n \leq C_{n,k}(\infty, \bbR_+) \, \|f\|_{L^{\infty}((0,\infty);dx)}^{n-k} 
\big\|f^{(n)}\big\|_{L^{\infty}((0,\infty);dx)}^k,   \no \\  
& \hspace*{5.45cm} k \in \{1, \dots, n-1\}, \; n \in \bbN, \; n \geq 2,   \no \\
& f\in \big\{g \in L^{\infty}((0,\infty);dx) \,\big| \, g, \dots, g^{(n-1)}\in AC([0,R])\text{ for all }R>0,  
\lb{2.9c} \\
& \hspace{6.75cm} g^{(n)}\in L^{\infty}((0,\infty);dx)\big\}    \no \\
& \hspace*{6mm} \quad = W^{n,\infty}((0,\infty)),      \no 
\end{align}
to higher-order analogs of the Kallman--Rota estimate \eqref{2.2} of the form, 
\begin{align}
\begin{split} 
\big\|A^k f \big\|_{\cX}^n \leq C_{n,k}(\infty, \bbR_+) \, 
\|f\|_{\cX}^{n-k} \big\|A^nf\big\|_{\cX}^k, \quad f\in\dom\big(A^n\big),&   \\
 k \in \{1, \dots, n-1\}, \; n \in \bbN, \; n \geq 2.&     \lb{2.9d} 
\end{split} 
\end{align}
For additional results in the higher-order cases, see, for instance, \cite{Ba98}, \cite{BE74}, \cite{Ch79}, 
\cite{GG75}, \cite{GR90}, \cite{Ku75}, \cite{KZ79a}, \cite{KZ91}, \cite[Ch.~1]{KZ92}, \cite{Lj64},  
\cite{Ne80}, \cite{Pr79}, \cite{Ra96}, \cite{SE89}, \cite{St65}, \cite{TW72}.  \\[1mm]
$(iii)$ For norm inequalities in connection with generators of cosine operator functions 
(cf.\ \cite[Sects.~3.14--3.16]{ABHN01}, \cite[Ch.~II]{Fa85}, \cite[Sect.~II.8]{Go85}) we refer, for instance, to \cite{LM10}, \cite{SE89}. The case of analytic semigroups (cf.\ \cite[Sect.~3.7]{ABHN01}, \cite[Sect.~2.5]{Da80}, \cite[Sect.~1.5]{Go85}) is treated in \cite{SE89}. \hfill$\diamond$
\end{remark}

Theorem \ref{t2.1} can be rewritten replacing the contraction semigroup by a uniformly bounded semigroup, that is, a semigroup such that for some $M \geq 1$,
\begin{equation} 
\|T(t)\|_{B(\cX)}\leq M, \quad t \in [0,\infty).
\end{equation}

\begin{corollary} \lb{c2.3} $($Pazy \cite[Lemma 2.8]{Pa83}$)$.
Let $c \in \bbC\backslash\{0\}$ and suppose that $cA$ generates a $C_0$ semigroup $T(t)$ satisfying $\|T(t)\|_{B(\cX)}\leq M$, $t \in [0,\infty)$, 
on a Banach space $\cX$. Then  
\begin{equation}
\|Af\|_{\cX}^2 \leq 4 \, M^2 \, \|f\|_{\cX} \big\|A^2f\big\|_{\cX}, \quad f\in\dom\big(A^2\big).
\end{equation}
\end{corollary}
\begin{proof}
One obtains this result by introducing the constant $M$ when bounding $T(t)$ in equation \eqref{2.9}, noting that $M\geq1$ as $T(0)=I_{\cX}$.
\end{proof}

\begin{example}\lb{e2.4}$($See, \cite[Example 9.10]{Go85}$)$.
Consider $A=d/dx$ in $L^p((0,\infty);dx)$, $p \in [1,\infty) \cup \{\infty\}$, and abbreviate 
$\bbR_+ = (0,\infty)$. Then 
\begin{align}
&\|f'\|_{L^p((0,\infty);dx)}^2 \leq C_{2,1}(p,\bbR_+) \, \|f\|_{L^p((0,\infty);dx)} \|f''\|_{L^p((0,\infty);dx)},  \no \\
& f\in\dom\bigg(\dfrac{d^2}{dx^2}\bigg)=\big\{g \in L^p((0,\infty);dx) \,\big| \, g,g'\in AC([0,R])\text{ for all }R>0,  
\lb{2.10} \\
& \hspace{8.1cm} g''\in L^p((0,\infty);dx)\big\}      \no \\
& \hspace*{2.15cm} \quad = W^{2,p}((0,\infty)),    \no 
\end{align}
where
\begin{align}
\begin{split} 
& C_{2,1}(1,\bbR_+) = 5/2, \quad C_{2,1}(2,\bbR_+) = 2, \\ 
& C_{2,1}(p,\bbR_+) \leq 4 
= C_{2,1}{(\infty,\bbR_+)}, \; p \in [1,\infty) \cup \{\infty\}.    \lb{2.11} \\
\end{split}
\end{align} 
Landau \cite{La13} proved the case $p=\infty$ in 1913; Hardy and Littlewood \cite{HL32} proved the case $p=2$ in 1932 $($cf.\ also \cite{Co77}, \cite[Theorem 2.2]{KZ92}$)$; Berdyshev \cite{Be71} proved the case $p=1$ in 1971. The best possible constant $C_{2,1}(p,\bbR_+)$ in \eqref{2.10} is not known otherwise, however, it is known to be a continuous function of $p$ $($cf. \cite[Theorem 2.19]{KZ92}$)$; for more details see \cite{FKKZ85}, \cite[Ch.~2]{KZ92}. 
\end{example}

Theorem \ref{t2.1}  can be refined in the case of Hilbert spaces as seen in the next result, the original proof of which is due to Kato \cite{Ka71}.

\begin{theorem}\lb{t2.5} $($Kato \cite{Ka71}, see also \cite[Theorem 9.9]{Go85}, \cite{KZ03}$)$.
Let $c \in \bbC\backslash\{0\}$ and suppose that $cA$ generates a $C_0$ contraction semigroup $T(t)$,  $t \in [0,\infty)$, on a Hilbert space $\cH$. Then  
\begin{equation}\lb{2.12}
\|Af\|_{\cH}^2 \leq 2 \, \|f\|_{\cH} \big\|A^2f\big\|_{\cH}, \quad f\in\dom\big(A^2\big), 
\end{equation}
and equality holds if and only if for some $D \in[0,\infty)$,
\begin{equation}
c^2 A^2f + D c Af + D^2f=0 \, \text{ and } \, \Re\big(\big(f, c^2 A^2f\big)_{\cH}\big)=0.    \lb{2.13}
\end{equation}
In the case of equality one necessarily has
\begin{equation}
|c|^2 \big\|A^2 f\big\|_{\cH} = D^2 \|f\|_{\cH}, \quad f\in\dom\big(A^2\big). 
\end{equation}

In fact, the hypothesis that $c A$ generates a $C_0$ contraction semigroup is unnecessary, it suffices to assume that 
\begin{equation}
\Re((f, cAf)_{\cH}) \leq 0, \quad f \in \dom(A).   \lb{2.18} 
\end{equation}
\end{theorem}
\begin{proof}[Sketch of proof.]
We briefly follow the proof provided in \cite{KZ03}, assuming without loss of generality that $c=1$. In the following, let $d > 0$ and $f \in \dom\big(A^2\big)$, then
\begin{equation}
\Re(([d A +I_{\cH}]f, d A [d A +I_{\cH}]f)_{\cH}) \leq 0   \lb{2.19} 
\end{equation}
implies
\begin{align}
\begin{split} 
& (d Af, d Af)_{\cH} + \big(\big[d^2 A^2 +d A + I_{\cH}\big]f, \big[d^2 A^2 +d A + I_{\cH}\big]f\big)_{\cH} \\
& \quad \leq \big(d^2A^2f, d^2A^2f\big)_{\cH} + \|f\|_{\cH}^2,   \lb{2.20} 
\end{split}
\end{align}
and hence
\begin{equation}
d^2 \|Af\|_{\cH}^2 \leq d^4 \big\|A^2 f\big\|_{\cH}^2 + \|f\|_{\cH}^2.   \lb{2.21} 
\end{equation}
Dividing \eqref{2.21} by $d^2$ and taking $d \to \infty$ yields that if $\big\|A^2 f\big\|_{\cH}=0$, then 
$ \|Af\|_{\cH} = 0$ and hence \eqref{2.12} holds. Otherwise, let $d^2 = \|f\|_{\cH} \big\|A^2 f\big\|_{\cH}^{-1}$ and then \eqref{2.21} once more yields \eqref{2.12}. 

Next, assume that for some $0 \neq f \in \dom\big(A^2\big)$, 
\begin{equation} 
\|Af\|_{\cH}^2 = 2 \, \|f\|_{\cH} \big\|A^2f\big\|_{\cH}.
\end{equation} 
If $A^2f = 0$ then $\|Af\|_{\cH} = 0$ and \eqref{2.13} holds with $D = 0$. If $A^2f \neq 0$, let 
$d^2 = \|f\|_{\cH} \big\|A^2 f\big\|_{\cH}^{-1}$ and inserting this into \eqref{2.20} implies 
$d^2 A^2 f + d Af + f =0$ and hence the first relation in \eqref{2.13} with $D = d^{-1}$. Working back from 
\eqref{2.20} to \eqref{2.19} yields
\begin{align}
\begin{split} 
0 &= \Re(([d A + I_{\cH}]f, d A [d A + I_{\cH}]f)_{\cH}) 
= \Re \big(\big([d A + I_{\cH}]f, \big[d^2A^2 + d A\big]f\big)_{\cH}\big)    \\
&= \Re \big(\big(d^2 A^2f, [-f]\big)_{\cH}\big) = - d^2 \, \Re \big(\big(f, A^2 f\big)_{\cH}\big),
\end{split} 
\end{align}
implying  the last relation in \eqref{2.13}.

Conversely, suppose that for some $D_0\in(0,\infty)$, 
\begin{align}\lb{2.28}
A^2f + D_0 Af + D_0^2f=0\text{ and }\Re\big(\big(f, A^2f\big)_{\cH}\big)=0.
\end{align}
From \eqref{2.28} one obtains
\begin{align}\lb{2.29}
\no  \|Af\|_{\cH}^2&= D_0^{-2}\big\|A^2f + D_0^2 f\big\|_{\cH}^2 
= D_0^{-2}\Big[\big\|A^2f\big\|_{\cH}^2 + D_0^4 \, \|f\|_{\cH}^2 
+ 2 D_0^2 \, \Re\big(\big(f, A^2f\big)_{\cH}\big) \Big]\\
&= D_0^{-2}\big\|A^2f\big\|_{\cH}^2 + D_0^2 \, \|f\|_{\cH}^2.
\end{align}
Since $\|Af\|_{\cH}^2 \leq 2 \, \|f\|_{\cH} \big\|A^2f\big\|_{\cH}$, it follows that
\begin{equation}
D_0^{-2} \big\|A^2f\big\|_{\cH}^2 - 2 \|f\|_{\cH} \big\|A^2f\big\|_{\cH} +D_0^2\|f\|_{\cH}^2\leq0,
\end{equation}
hence $\big[D_0^{-1}\big\|A^2f\big\|_{\cH} - D_0\|f\|_{\cH}\big]^2 \leq 0$,
or $D_0^{-1} \big\|A^2f\big\|_{\cH} = D_0 \, \|f\|_{\cH}$, equivalently, 
\begin{equation}\lb{2.32}
\big\|A^2 f\big\|_{\cH} = D_0^2 \, \|f\|_{\cH}.
\end{equation}
Substituting \eqref{2.32} into \eqref{2.29} yields equality in \eqref{2.12}. 
\end{proof}

The following result is a natural extension of Example \ref{e2.4} in the case $p=2$ and it illustrates a particular case of Theorem \ref{t2.5}. While this observation is not new, we record its proof due to its simplicity:

\begin{lemma} \lb{l2.6} $($Ljubi{\v c} \cite[Theorem~5]{Lj64}, see also Chernoff \cite{Ch79} and Kwong and Zettl \cite{KZ79b}$)$.
Suppose $S$ is a maximally symmetric, non-self-adjoint operator in a Hilbert space $\cH$. 
Let $A = c \, S^*$ for some $c \in \bbC\backslash\{0\}$. Then 
\begin{equation}
\|Af\|_{\cH}^2\leq 2 \, \|f\|_{\cH} \big\|A^2f\big\|_{\cH}, \quad f\in\dom\big(A^2\big).     \lb{2.33} 
\end{equation} 
The constant $2$ is optimal and the case of equality in \eqref{2.33} is determined as in \eqref{2.13} applied to 
$\pm i S^*$. 
\end{lemma}
\begin{proof}
Since $S$ is maximally symmetric yet non-self-adjoint, one of its deficiency indices equals zero, that is, 
\begin{equation}
\dim(\ker(S^* - i I_{\cH})) = 0, \, \text{ or } \,  \dim(\ker(S^* + i I_{\cH})) = 0,
\end{equation} 
but not both are zero simultaneously. Consequently (cf.\ \cite[Sect.~V.3.4]{Ka80}), the spectrum of $S$ is either the closed complex 
upper half-plane or the closed complex lower half-plane, respectively,
\begin{equation}
\sigma(S) = \ol{\bbC_+}  \, \text{ or } \,  \sigma(S) =\ol{\bbC_-}.
\end{equation}  
Thus, 
\begin{equation} 
\sigma(i S^*) = \{z \in \bbC \, | \, \Re(z) \leq 0\},  \, \text{ or } \,  
\sigma(i S^*) = \{z \in \bbC \, | \, \Re(z) \geq 0\}, 
\end{equation}
and hence one of $\pm i S^*$ is m-accretive and thus generates a contraction semigroup on $\cH$. (Actually, one of them satisfies \eqref{2.18} which suffices for our argument.) Since 
$A = c \, S^*$, the estimate \eqref{2.33} (and the case of equality therein) now follows from \eqref{2.12} (and \eqref{2.13}). That the constant $2$ is optimal is a consequence of Example \ref{e2.4} with $p=2$.
\end{proof}

\begin{remark} \lb{r2.7} 
$(i)$ We emphasize that the constants to be found in Theorems \ref{t2.1}  and \ref{t2.5} are optimal as can be seen from Example \ref{e2.4}. The equality condition given in Theorem \ref{t2.5}, in the case $p=2$ in Example \ref{e2.4} then becomes 
\begin{align}
f'' + D f' + D^2f=0 \, \text{ and } \, \Re((f, f'')_{L^2((0,\infty);dx)}) 
= \Re\bigg(\int_0^\infty dx\, \ol{f(x)} f''(x)\bigg)=0. 
\end{align}
Hardy and Littlewood \cite{HL32} $($see also \cite[p.~187--188]{HLP88}$)$ showed these conditions determine $f$ to be
\begin{equation}
f_{C,D}(x)= C \, e^{- D x/2}\sin\big(3^{1/2} D (x/2) - (\pi/3)\big), 
\quad C \in\bbC,\; D > 0, \; x \geq 0.
\end{equation}
$(ii)$ Kato's original proof of Theorem \ref{t2.5} relied on properties of the Cayley transform of $A$. For additional results in this direction see, \cite{Bo73}, \cite{Bo73a}, \cite{Ch79}, \cite{Ev71}, \cite{GG81}, \cite{KZ79b}, 
\cite{Pa81}, \cite{Ph81}. 
\\[1mm] 
$(iii)$ For Theorem \ref{t2.5} under weaker hypotheses on $A$ see \cite{HO16}. \\[1mm] 
$(iv)$ Higher-order inequalities of the type \eqref{2.9d} with explicit constants in the Hilbert space context were derived by Protter \cite{Pr79}, thus extending Kato's result (where $k=1$, $n=2$) to the higher-order case. Additional results appeared in \cite{Ch79}, \cite{KZ79b}, \cite[Theorem~7]{Lj64}, \cite{Ph81}, in particular, these authors prove the following result: If $C_{n,k}(2,\bbR_+)$ denote the constants in the $L^2$-inequalities, 
\begin{align}
& \big\|f^{(k)}\big\|_{L^2((0,\infty);dx)}^n \leq C_{n,k}(2,\bbR_+) \, \|f\|_{L^2((0,\infty);dx)}^{n-k} 
\big\|f^{(n)}\big\|_{L^2((0,\infty);dx)}^k,    \no \\  
& \hspace*{4.95cm} k \in \{1, \dots, n-1\}, \; n \in \bbN, \; n \geq 2,   \no \\
& f\in \big\{g \in L^2((0,\infty);dx) \,\big| \, g, \dots, g^{(n-1)}\in AC([0,R])\text{ for all }R>0,  
\lb{2.35} \\
& \hspace{6.8cm} g^{(n)}\in L^2((0,\infty);dx)\big\}   \no \\
& \hspace*{6mm} \quad = H^n((0,\infty)),    \no 
\end{align} 
then the same constants feature in the corresponding higher-order inequalities involving the operator $A$ as follows: {\it Suppose that $A$ is a nonzero constant multiple of a generator of a contraction semigroup $T(t)$, $t \in [0,\infty)$ $($equivalently, a nonzero constant multiple of a densely defined dissipative operator\,$)$, on a Hilbert space $\cH$, then} 
\begin{align}
\begin{split} 
\big\|A^k f \big\|_{\cH}^n \leq C_{n,k}(2,\bbR_+) \, \|f\|_{\cH}^{n-k} \big\|A^nf\big\|_{\cH}^k, \quad f\in\dom\big(A^n\big),&   \\
 k \in \{1, \dots, n-1\}, \; n \in \bbN, \; n \geq 2.&     \lb{2.36} 
\end{split} 
\end{align}

We also recall the symmetry property \cite{Ch79}, \cite{KZ79b}, \cite[Corollary~1, p. 71]{Lj64}, 
\begin{equation}
C_{n,k}(2,\bbR_+) = C_{n,n-k}(2,\bbR_+), \quad k \in \{1, \dots, n-1\}, \; n \in \bbN, \; n \geq 2
\end{equation}
(a property not shared by $C_{n,k}(\infty,\bbR_+)$).

In particular, the extension of Lemma \ref{l2.6} to the case of the higher-order inequalities \eqref{2.36} (under the same hypothesis that $A = c \, S^*$, for $S$ maximally symmetric and non-self-adjoint, 
$c \in \bbC\backslash\{0\}$) has been observed by Ljubi{\v c} \cite[Theorem~5]{Lj64} in 1964, and was subsequently rederived by rather different means in Chernoff \cite{Ch79} and Kwong and Zettl \cite{KZ79b}. 
\hfill$\diamond$
\end{remark}

\section{Norm Inequalities for Generators of $C_0$ Groups} \lb{s3} 

As all these results so far have concerned semigroups, it is natural for one to ask what the corresponding results are when considering groups rather than semigroups.  

\begin{theorem}\lb{t3.1} $($See, \cite{Di75}, \cite{GR90}, \cite{KR70}$)$.
Let $c \in \bbC\backslash\{0\}$ and suppose that $cA$ generates a $C_0$ group $T(t)$, $t\in\bbR$, of isometries on a Banach space $\cX$. Then 
\begin{equation}
\|Af\|_{\cX}^2 \leq 2 \, \|f\|_{\cX} \big\|A^2f\big\|_{\cX}, \quad f\in\dom\big(A^2\big).    \lb{3.1}
\end{equation}
\end{theorem}
\begin{proof}[Sketch of proof.]
The two proofs mentioned in Remark \ref{r2.2} extend to the present case. For instance (choosing again $c=1$ for brevity), the standard semigroup proof in \cite[2nd ed., Theorem 1.1, p. 237]{Go85} proceeds  in establishing the Taylor-type formula  
\begin{equation}\lb{3.2}
T(-t)f=f-tAf+\int_{-t}^0ds\ (t+s)T(s)A^2f,\quad t \in [0,\infty), \; f\in\dom\big(A^2\big), 
\end{equation}
and from this one obtains
\begin{equation}\lb{3.3}
\|Af\|_{\cX}\leq \dfrac{1}{t} \, \|f\|_{\cX}+\dfrac{t}{2} \,\big\|A^2f\big\|_{\cX},\quad t > 0.
\end{equation}
If $A^2f=0$, letting $t\uparrow\infty$ in \eqref{3.3} shows that $Af=0$. If $A^2f\neq0$, one minimizes the right-hand side of \eqref{3.3} over $t > 0$ by noting that the minimum occurs at
$t= 2^{1/2} \|f\|_{\cX}^{1/2}\big/\big\|A^2f\big\|_{\cX}^{1/2}$. Substituting this minimum into \eqref{3.3} yields \eqref{3.1}.

Similarly, the functional analytic proof of Certain and Kurtz \cite{CK77} extends to the present estimate \eqref{3.1} upon systematically replacing the half-line $[0,\infty)$ by $\bbR$. 
\end{proof}

\begin{remark} \lb{r3.2} 
$(i)$ For additional references in this context see \cite{GG75}, \cite{GKZ83}, \cite{KZ79}, \cite{KZ79a}, 
\cite{KZ80}, \cite{KZ81}, \cite{KZ81a}, \cite{KZ84}, \cite[Ch.~2]{KZ92}. \\[1mm]
$(ii)$ The case of higher-order inequalities on $\bbR$ was studied by Kolmogorov \cite{Ko39}, and again 
the strategy of proof in Certain--Kurtz \cite{CK77} extends the analog of \eqref{2.9c} on $\bbR$ with smaller constants $C_{n,k}(\infty,\bbR)$ to the analog of the higher-order Kallman--Rota inequalities \eqref{2.9d} with the same optimal constants $C_{n,k}(\infty,\bbR)$. For additional results in the higher-order case, see, \cite{Ba98}, \cite{GG75}, \cite{KS81}, \cite{KZ91}, \cite[Ch.~1]{KZ92}, \cite{Lj64}, \cite{Pr79}, \cite{SE89}, \cite{St65}, \cite{St57}. \\[1mm]
$(iii)$ In the special case where $\cX$ is a Hilbert space $\cH$ in Theorem \ref{t3.1}, the constant $2$ on the right-hand side of \eqref{3.1} can be replaced by $1$ (cf., \cite{Ch79}). Indeed, in this case the group of isometries, $T(t)$, $t \in \bbR$, becomes a unitary group, hence its generator $A$ is a skew-adjoint operator by Stone's theorem and then the assertion 
\begin{equation}
\|Af\|_\cH^2 \leq \|f\|_{\cH} \big\|A^2f\big\|_{\cH}, \quad f\in\dom\big(A^2\big), 
\end{equation}
is a special case of Lemma \ref{l4.3}. \hfill$\diamond$
\end{remark}

Once more the constant in Theorem \ref{t3.1} is optimal as can be seen from the next example.

\begin{example}\lb{e3.3}
Consider $A=d/dx$ in $L^p(\bbR;dx)$, $p \in [1, \infty) \cup \{\infty\}$. Then 
\begin{align}
&\|f'\|_{L^p(\bbR;dx)}^2\leq C_{2,1}(p,\bbR) \, \|f\|_{L^p(\bbR;dx)} \|f''\|_{L^p(\bbR;dx)}, \no \\
& f\in\dom\bigg(\dfrac{d^2}{dx^2}\bigg)=\big\{f\in L^p(\bbR;dx) \, \big| \, f,f'\in AC([0,R])\text{ for all }R>0, 
\lb{3.6} \\
& \hspace{8.1cm}f''\in L^p(\bbR;dx)\big\}   \no \\
& \hspace*{2.15cm} \quad = W^{2,p}(\bbR),    \no 
\end{align}
where 
\begin{align}
\begin{split} 
&C_{2,1}(1,\bbR) = 2, \quad C_{2,1}(2,\bbR)  =1, \\ 
& C_{2,1}(p,\bbR) \leq 2= C_{2,1}(\infty, \bbR), \quad C_{2,1}(p,\bbR) \leq C_{2,1}(p,\bbR_+), \; 
 p \in [1, \infty) \cup \{\infty\}. 
 \end{split}
\end{align} 
Landau \cite{La13} in 1913, and Hadamard \cite{Ha14} in 1914 proved the case $p=\infty$; Hardy and Littlewood \cite{HL32} proved the case $p=2$ in 1932 $($cf.\ also \cite[Theorem 2.2]{KZ92}$)$; Berdyshev \cite{Be71} proved the case $p=1$ in 1971. The best possible constant $C_{2,1}(p,\bbR)$ in \eqref{3.6} is not known otherwise, however, it is again known to be a continuous function of $p$ $($cf. \cite[Theorem 2.19]{KZ92}$)$; for more details see \cite{FKKZ83}, \cite[Ch.~2]{KZ92}. In particular, Hardy and Littlewood \cite{HL32} showed that equality holds in \eqref{3.6} when $p=2$ only when $f\equiv 0$ and the best constant can be proven by taking $f$ of the form,
\begin{equation}
f(x)=\begin{cases}\sin (x), &|x|<n\pi,\\
0, & |x|>n\pi,
\end{cases} \quad n \in \bbN, 
\end{equation}
and smoothing near $x=\pm n\pi$, $n \in \bbN$, so as to ensure $f''$ is continuous $($see also 
\cite[p.~193]{HLP88}$)$. For $p=\infty$, inequality \eqref{3.6} was generalized to higher-order derivatives on the real line by Kolmogorov in \cite{Ko39}, who showed that the sharp constants could be expressed through relations with the so-called Favard constants $($cf. \cite{FKKZ85}$)$.
\end{example}

We refer to Kwong and Zettl \cite{KZ92} for an extensive treatment of the topic of $L^p$-norm inequalities for derivatives.

\section{Inequalities for Fractional Powers of Generators of Contraction Semigroups} \lb{s4} 

We now transition to the related topic of fractional powers of generators of contraction semigroups on Banach spaces. 

We start with the simple case of nonnegative, self-adjoint (generally, unbounded)  operators in a Hilbert space.  
The following elementary inequality is derived in the book by Krasnosel'skii, Pustylnik, Sobolevskii, and Zabreiko \cite[p.~223]{KPSZ76}:  

\begin{theorem}\lb{t4.1} $($See, \cite[Theorem 12.1, p. 223]{KPSZ76}$)$.
Let $S$ be a nonnegative self-adjoint operator on a Hilbert space $\cH$. Then for each $\tau\in(0,1)$, 
\begin{equation}\lb{SA ineq}
\big\|S^\tau f\big\|_{\cH} \leq \|f\|_{\cH}^{1-\tau} \|Sf\|_{\cH}^\tau, \quad f\in\dom(S).
\end{equation}
\end{theorem}
\begin{proof}
Using the spectral representation of the operator $S$ (and observing $\sigma(S) \subseteq [0,\infty)$), 
\begin{equation}
S=\int_{\sigma(S)}\ \lambda \, dE_S(\lambda),
\end{equation}
implies
\begin{equation}
\|S^\tau f\|_{\cH}^2=\int_{\sigma(S)}\ \lambda^{2\tau} \, d\|E_S(\lambda)f\|_{\cH}^2,\quad f\in\dom(S^{\tau}).
\end{equation}
It follows from H\"{o}lder's inequality that
\begin{align}
\no  \big\|S^\tau f \big\|_{\cH}^2&=\int_{\sigma(S)}\ \lambda^{2\tau} \, d\|E_S(\lambda)f\|_{\cH}^2\\
&\leq \bigg(\int_{\sigma(S)}\ \lambda^{2} \, d\|E_S(\lambda)f\|_{\cH}^2\bigg)^\tau \bigg(\int_{\sigma(S)} 
\, d\|E_S(\lambda)f\|_{\cH}^2\bigg)^{1-\tau}\\
\no  &=\|Sf\|_{\cH}^{2\tau}\|f\|_{\cH}^{2(1-\tau)}, \quad f \in \dom(S).
\end{align}
\end{proof}

\begin{remark} \lb{r4.2} 
$(i)$ As the elementary example $S=I_{\cH}$ shows, the constant $1$ on the right-hand side of \eqref{SA ineq} is optimal. \\[1mm]  
$(ii)$ One compares Theorem \ref{t4.1}  to Theorem \ref{t2.5} as follows. Choose a self-adjoint operator $S$ such that $S=A^2$,  $A \geq 0$, then also $S \geq 0$ in $\cH$. Theorem \ref{t4.1} then implies,
\begin{equation}
\big\|A^{2\tau} f\big\|_{\cH} \leq \|f\|_{\cH}^{1-\tau} \big\|A^2f\big\|_{\cH}^\tau,
\quad f \in \dom\big(A^2\big), 
\end{equation}
and choosing $\tau=1/2$,
\begin{equation}\lb{4.8}
\|Af\|_{\cH} \leq \|f\|_{\cH}^{1/2} \big\|A^2f\big\|_{\cH}^{1/2}, 
\quad f \in \dom\big(A^2\big).
\end{equation}
Equivalently, this yields for any nonnegative $($resp., nonpositive$)$ self-adjoint operator $A$ a in a Hilbert space $\cH$ that 
\begin{equation}
\|Af\|_{\cH}^2 \leq \|f\|_{\cH} \big\|A^2f\big\|_{\cH}, \quad f\in\dom\big(A^2\big).  \lb{4.9} 
\end{equation}
Comparing \eqref{4.9} to \eqref{2.12}, one notes that the optimal constant has now been reduced from $2$ to $1$ in inequality \eqref{2.12}. \hfill$\diamond$
\end{remark}

Actually, as pointed out in \cite{Lj64} (see also \cite{KZ03}), self-adjointness and nonnegativity of $A$ in \eqref{4.9} are not needed at all. Indeed, it suffices to assume that $A$ is a constant multiple of a symmetric operator:

\begin{lemma}\lb{l4.3}
Let $c \in \bbC\backslash\{0\}$ and suppose that $cA$ is a symmetric operator in a Hilbert space $\cH$. Then 
\begin{equation}
\|Af\|_{\cH}^2 \leq \|f\|_{\cH} \big\|A^2f\big\|_{\cH}, \quad f\in\dom\big(A^2\big),     \lb{4.10} 
\end{equation}
with equality in \eqref{4.10} if and only if $A^2f = C f$ for some $C \in \bbC$. 
\end{lemma}
\begin{proof}
Let $f\in\dom\big(A^2\big)$, then Cauchy's inequality implies
\begin{equation}
\|c Af\|_{\cH}^2 = (c Af,c Af)_{\cH} = \big(f,c^2 A^2f\big)_{\cH} \leq |c|^2 \|f\|_{\cH} \big\|A^2f\big\|_{\cH}.
\end{equation}
\end{proof}

\begin{remark}\lb{r4.4}
Despite the most elementary nature of the proof of Lemma \ref{l4.3}, it should be noted that Naimark \cite{Na40} proved the existence of a symmetric operator $A$ in $\cH$ such that $\dom\big(A^2\big) = \{0\}$, rendering  \eqref{4.10} vacuous in this case. (Of course, this pathology cannot occur in the context of \eqref{4.9}.) In this context we also refer to Schm\"udgen \cite{Sc83}. \hfill $\diamond$ 
\end{remark}

More generally, fractional powers of the 1st-order differentiation operator $D = d/dx$ on $C_b([0,\infty))$ (the space of bounded, uniformly continuous functions on $[0,\infty)$) and associated norm inequalities were already proved in 1935 by Hardy, Landau, and Littlewood \cite{HLL35} employing the Riemann--Liouville formula for fractional derivatives 
\begin{align}
\begin{split} 
\big((- D)^\alpha f\big)(x) = - \Gamma(- \alpha)^{-1} \int_0^{\infty} dt \, t^{-\alpha - 1} [f(t) - f(x+t)],&   \\ 
0 < \alpha < 1, \; f \in C_b([0,\infty)).&
\end{split}  
\end{align}

Assuming that $A$ generates a contraction semigroup $T(t)$, $t \in [0,\infty)$, on the Banach space $\cX$, one can define fractional powers of $- A$ via the formula (see, \cite{BBW68}, \cite[Theorem~6.1.6]{MS01})
\begin{align}
& (- A)^{\gamma} f = C_{\gamma,k}^{-1} \int_0^{\infty} dt \, t^{- \gamma - 1} [I_{\cH} - T(t)]^k f, \quad 
 f \in \cX,    \\
& C_{\gamma,k} = \int_0^{\infty} dt \, t^{- \gamma - 1} \big[1 - e^{-t}\big]^k, \quad 
k - 1 \leq \gamma < k, \; k \in \bbN.
\end{align}

Chernoff \cite{Ch79} proves the following two results:

\begin{theorem} \lb{t4.5} $($Chernoff \cite{Ch79}$)$. 
Suppose that $A$ generates a $C_0$ contraction semigroup $T(t)$,  $t \in [0,\infty)$, on a Banach space $\cX$ and let $0 < \alpha < \beta$. Then  
\begin{equation}
\big\|(-A)^{\alpha} f\big\|_{\cX}^{\beta} \leq C_{\beta,\alpha}(\infty,\bbR_+) \, \|f\|_{\cX}^{\beta - \alpha} 
\big\|(-A)^\beta f\big\|_{\cX}^{\alpha}, \quad f\in\dom\big(A^{\beta}\big).     \lb{4.13} 
\end{equation}
\end{theorem}

In the Hilbert space context this result turns into the following:

\begin{theorem} \lb{t4.6} $($Chernoff \cite{Ch79}$)$. 
Suppose that $A$ generates a $C_0$ contraction semigroup $T(t)$,  $t \in [0,\infty)$, on a Hilbert  space $\cH$ and let $0 < \alpha < \beta$. Then  
\begin{equation}
\big\|(-A)^{\alpha} f\big\|_\cH^{\beta} \leq C_{\beta,\alpha}(2,\bbR_+) \, \|f\|_\cH^{\beta - \alpha} 
\big\|(-A)^\beta f\big\|_{\cH}^{\alpha}, \quad f\in\dom\big(A^{\beta}\big).     \lb{4.14} 
\end{equation}
\end{theorem}
\noindent 
Here the constants $C_{\beta,\alpha}(\infty,\bbR_+), C_{\beta,\alpha}(2,\bbR_+)$ are the same as in the context of fractional powers of the maximally defined 1st-order differentiation operator in $L^{\infty}((0,\infty);dx)$ and $L^2((0,\infty);dx)$, respectively. 

The analogous theorems hold in case $A$ generates a group of isometries $T(t)$, $t \in \bbR$, employing the constants $C_{\beta,\alpha}(\infty,\bbR), C_{\beta,\alpha}(2,\bbR)$, instead.

\section{On Some Extensions of the Hardy--Littlewood Inequality} \lb{s5} 

We finally discuss some extensions of the Hardy--Littlewood inequality in the context of $L^2((0,\infty);dx)$ and $L^2(\bbR;dx)$. 

The following extensions of the Hardy--Littlewood inequality appeared in  \cite{BE73} (see also \cite{BP84}, \cite{Co77a}, \cite{Ev75}),  
\begin{align}
& \big| \|f'\|_{L^2((0,\infty);dx)}^2 - \mu \|f\|_{L^2((0,\infty);dx)}^2\big| \leq 2 \, \|f\|_{L^2((0,\infty);dx)} 
\| [f'' + \mu f]\|_{L^2((0,\infty);dx)},   \no \\
& \hspace*{6.9cm} \mu \in [0,\infty), \; f \in H^2((0,\infty)),    \lb{5.1} \\
\begin{split}
& \big| \|f'\|_{L^2(\bbR;dx)}^2 - \mu \|f\|_{L^2(\bbR;dx)}^2\big| \leq \|f\|_{L^2(\bbR;dx)} 
\|[f'' + \mu f]\|_{L^2(\bbR;dx)},   \\
& \hspace*{6.8cm}  \mu \in \bbR, \; f \in H^2(\bbR).     \lb{5.2}
\end{split} 
\end{align}
Both estimates are strict, that is, equality holds in \eqref{5.1} or \eqref{5.2} if and only if $f \equiv 0$.

The estimate \eqref{5.1} was abstracted in \cite{Ch79} as follows: 

\begin{theorem} \lb{t5.1} $($Chernoff \cite{Ch79}$)$.  
Let $c \in \bbC \backslash \{0\}$ and suppose that $c A$ generates a $C_0$ contraction semigroup $T(t)$, $t \in [0,\infty)$, on a Hilbert space $\cH$. Then 
\begin{equation}
\big|\|A f\|_{\cH}^2 - \mu \|f\|_{\cH}^2\big| \leq 2 \, \|f\|_{\cH} \big\|\big[A^2 + \mu I_{\cH}\big]f\big\|_{\cH}, 
\quad \mu \in [0,\infty), \; f \in \dom\big(A^2\big). 
\end{equation}
\end{theorem}

A further extension of the Hardy--Littlewood inequality, naturally involving quadratic forms, was initiated by Everitt \cite{Ev71} in 1971 originally in the context of regular Sturm--Liouville problems on the half-line 
$(0,\infty)$. To describe this problem one needs a few preparations.

Suppose that $- \infty < a < b \leq \infty$, and assume that for all $c \in (a,b)$, 
\begin{align}
& p^{-1} \in L^1((a,c);dx), \; \text{$p > 0$ a.e.~on $(a,b)$},    \no \\
& q \in L^1((a,c);dx), \; \text{$q$ real-valued a.e.~on $(a,b)$},    \lb{5.4} \\
& r \in L^1((a,c);dx), \; \text{$r > 0$ a.e.~on $(a,b)$}.    \no 
\end{align}
Introducing the Sturm--Liouville differential expression (regular at the point $a$) 
\begin{equation}
\tau = r(x)^{-1}\bigg[ - \frac{d}{dx} p(x) \frac{d}{dx} + q(x)\bigg], \quad x \in (a,b), 
\end{equation}
assuming $\tau$ to be in the limit point case at the point $b$ (i.e., for all $c \in (a,b)$ and for all $z \in \bbC$, 
$\tau \psi = z \psi$, has at least one solution $\psi(z, \dott)$ not in $L^2((c,b); r(x)dx)$), one is interested in $L^2((a,b); r(x)dx)$-realizations associated with $\tau$ and introduces minimal and maximal operators corresponding to $\tau$ via
\begin{align}
& (T_{min} f)(x) = (\tau f)(x), \quad x \in (a,b),    \no \\
& f \in \dom(T_{min}) = \big\{g \in L^2((a,b); r(x)dx) \, \big| \, g, g^{[1]} \in AC([a,c]) \text{ for all $c \in (a,b)$}; \no \\
& \hspace*{4.3cm} g(a)=0=(pg')(a); \, \tau g \in L^2((a,b); r(x)dx)\big\}, \\
& (T_{max} f)(x) = (\tau f)(x), \quad x \in (a,b),    \no \\
& f \in \dom(T_{max}) = \big\{g \in L^2((a,b); r(x)dx) \, \big| \, g, g^{[1]} \in AC([a,c]) \text{ for all $c \in (a,b)$}; \no \\
& \hspace*{7.5cm}  \tau g \in L^2((a,b); r(x)dx)\big\}. 
\end{align}
Here we employed the notion of a quasi-derivative $g^{[1]}(x) = p(x) g'(x)$, $x \in (a,b)$.  
One then infers that $T_{min}$ is closed, densely defined, and symmetric and 
\begin{equation}
T_{min}^* = T_{max}, \quad T_{max}^* = T_{min}. 
\end{equation}
The problem posed by Everitt (in a slightly extended form) then reads as follows: Is there a constant 
$K_{\cD} = K_{\cD} ((a,b);p,q,r) \in (0,\infty)$ such that 
\begin{align}
\begin{split} 
& \bigg|\int_a^b dx \, \big[p(x) |f'(x)|^2 + q(x) |f(x)|^2\big]\bigg|  \lb{5.9} \\
& \quad \leq K_{\cD} \bigg(\int_a^b r(x) dx \, |f(x)|^2\bigg)^{1/2}  
\bigg(\int_a^b r(x) dx \, |(\tau f)(x)|^2\bigg)^{1/2}, \quad f \in \cD, 
\end{split} 
\end{align}
equivalently, 
\begin{align}
\begin{split} 
& \bigg|\int_a^b dx \, \Big[p(x)^{-1} \big|f^{[1]}(x)\big|^2 + q(x) |f(x)|^2\Big]\bigg|   \lb{5.10} \\
& \quad \leq K_{\cD} \|f\|_{L^2((a,b); r(x)dx)} \|\tau f\|_{L^2((a,b); r(x)dx)}, \quad f \in \cD, 
\end{split} 
\end{align}
for an appropriate linear subspace $\cD$ of $L^2((a,b); r(x)dx)$ satisfying 
\begin{equation}
\dom(T_{min}) \subseteq \cD \subseteq \dom(T_{max}). 
\end{equation}
It is implicitly assumed that both terms are finite on the left-hand side of \eqref{5.9}, \eqref{5.10}. 
Introducing the sesquilinear form 
\begin{align}
q(f,g) =  \int_a^b dx \, \Big[p(x)^{-1} \ol{f^{[1]}(x)} g^{[1]}(x) + q(x) \ol{f(x)} g(x)\Big], \quad f \in \cD,
\lb{5.12} 
\end{align}
(still implicitly assuming that both terms on the right-hand side of \eqref{5.12} are finite),  one recognizes that Everitt's problem is equivalent to the existence of $K_{\cD} = K_{\cD} ((a,b);p,q,r) \in (0,\infty)$ such that the integral inequality
\begin{equation}
|q(f,f)| \leq K_{\cD} \|f\|_{L^2((a,b); r(x)dx)} \|\tau f\|_{L^2((a,b); r(x)dx)}, \quad f \in \cD,    \lb{5.13}
\end{equation}
holds. In particular, the special case 
\begin{align}
& a = 0, \; b = \infty, \quad p(x) = r(x) =1, \quad q(x) = 0, \quad \tau = - \frac{d^2}{dx^2}, \quad x \in (0,\infty), 
\lb{5.14} \\
& \cD = H^2((0,\infty)), \quad K_{H^2((0,\infty))}(\bbR_+; 1,0,1) = 2,    \lb{5.15} 
& \intertext{or}   
& \cD = H_0^2((0,\infty)), \quad K_{H_0^2((0,\infty))}(\bbR_+; 1,0,1) = 1,     \lb{5.16} 
\end{align}
yields the classical Hardy--Littlewood inequality \cite{HL32},
\begin{equation}
\|f'\|_{L^2((0,\infty); dx)}^2 \leq 2 \, \|f\|_{L^2((0,\infty); dx)} \|f''\|_{L^2((0,\infty); dx)}, 
\quad f \in H^2((0,\infty)),    \lb{5.17} 
\end{equation}
(cf.\ \eqref{2.10}, \eqref{2.11} for $p=2$) and 
\begin{equation}
\|f'\|_{L^2((0,\infty); dx)}^2 \leq \|f\|_{L^2((0,\infty); dx)} \|f''\|_{L^2((0,\infty); dx)}, 
\quad f \in H_0^2((0,\infty)),     \lb{5.18} 
\end{equation}
as a special case of Lemma \ref{l4.3} since 
\begin{equation}
(A_0 f)(x) = f'(x), \; x \in (0,\infty), \quad  f \in \dom(A_0) = H_0^{1} ((0,\infty)),     \lb{5.19} 
\end{equation}
in $L^2((0,\infty); dx)$ is skew-symmetric (i.e., $i A_0$ is symmetric) and 
\begin{equation}
\big(A_0^2 f\big)(x) = f''(x), \; x \in (0,\infty), \quad f \in \dom\big(A_0^2\big) = H_0^{2} ((0,\infty)).  \lb{5.20} 
\end{equation}

The precise analysis of the problem stated in connection with \eqref{5.13}, originally due to Everitt \cite{Ev71},  is rather complex, utilizing results from the calculus of variations and from Weyl--Titchmarsh theory. This was continued in Bennewitz \cite{Be84}, Evans and Everitt \cite{EE82}, Evans and Zettl \cite{EZ78a}, and Ph\'ong \cite{Ph81} under more general hypotheses on $p, q,r$ employing, in addition, von Neumann's first formula for deficiency spaces (so adding some operator theory to this circle of ideas). To describe its solution we now follow the discussion in \cite{EE82}. 

Let $\bbC_{\pm} = \{\z \in \bbC \, | \, \pm \Im(z) > 0\}$, and suppose that 
$\theta_{\pi/2}(z, \dott), \phi_{\pi/2}(z, \dott)$ are solutions of $\tau u = z u$ (entire w.r.t. $z$) satisfying the initial conditions
\begin{align}
\begin{split}
& \theta_{\pi/2}(z, a) = 0, \quad  \;\;\;\, \theta_{\pi/2}^{[1]}(z, a) = 1,    \\
&  \phi_{\pi/2}(z, a) = -1, \quad  \phi_{\pi/2}^{[1]}(z, a) = 0,
\end{split}
\end{align}
and $\psi_{\pi/2,\pm}(z, \dott)$ is a (Neumann-type) Weyl--Titchmarsh solution
\begin{equation}
\psi_{\pi/2,\pm}(z,\dott) = \theta_{\pi/2}(z, \dott) + m_{\pi/2,\pm}(z) \phi_{\pi/2}(z, \dott) \in L^2((a,b); r(x)dx), 
\quad z \in \bbC_{\pm},
\end{equation}
with $m_{\pi/2,\pm}(\dott)$ the (Neumann-type) Weyl--Titchmarsh $m$-function analytic in $\bbC_{\pm}$, such that
\begin{equation}
\ol{m_{\pi/2,\pm} (z)} = m_{\pi/2,\mp}({\ol z}), \quad z \in \bbC_{\pm}.  
\end{equation}
Moreover, denoting $z = \rho e^{i \vartheta}$, $\rho \in (0,\infty)$, $\vartheta \in [0, 2\pi)$, one introduces
\begin{align}
& L_+(\vartheta) = \big\{\rho e^{i \vartheta} \, \big| \, \rho \in (0,\infty)\big\}, \quad  
L_-(\vartheta) = \big\{\rho e^{i (\vartheta + \pi)} \, \big| \, \rho \in (0,\infty)\big\}, \quad \vartheta \in (0, \pi/2], \no \\
& \vartheta_{\pm} = \inf \big\{\vartheta \in (0,\pi/2] \, \big| \, \text{for all } \varphi \in [\vartheta,\pi/2], 
\mp \Im\big(z^2 m_{\pi/2,\pm}(z)\big) \geq 0, \, z \in L_{\pm}(\varphi)\big\},   \no \\
& \vartheta_0 = \max \{\vartheta_+, \vartheta_-\} \in [0,\pi/2],      \\
& E_{\pm} = \big\{\rho \in (0,\infty) \, \big| \, z \in L_{\pm}(\vartheta_0), \, 
\Im\big(z^2 m_{\pi/2,\pm}(z)\big) = 0\big\},    \no \\
& Y_{\pm}(\rho,x) = \Im(z \psi_{\pi/2,\pm}(z,x)), \quad 
z \in L_{\pm}(\vartheta_0), \; x \in [a,b).    \no 
\end{align}
Finally, $\tau$ is said to be in the {\it strong limit point case at $b$} if 
\begin{equation}
\lim_{x \uparrow b} f(x) g^{[1]}(x) = 0, \quad f, g \in \dom(T_{max})      \lb{5.25}
\end{equation}
(see, e.g., \cite{Hi74} for sufficient conditions on $p,q,r$ guaranteeing the strong limit point endpoint). We recall that the strong limit point property at $b$ implies the limit point property at $b$ (cf.\ \cite{EE82}). 

Given these preparations, the principal result in \cite{EE82} (see also \cite{EE91}) reads as follows:

\begin{theorem} \lb{t5.2} $($Evans and Everitt \cite{EE82}$)$.  
Assume the conditions \eqref{5.4} and suppose that $\tau$ is in the strong limit point case at $b$. Given the preparations in \eqref{5.4}--\eqref{5.25}, the following assertions $(i)$--$(iv)$ hold: \\[1mm] 
$(i)$ $\vartheta_0 \in (0,\pi/2]$ $($in particular, $\vartheta_0 > 0$$)$. \\[1mm]
$(ii)$ The inequality
\begin{align}
\begin{split} 
& \bigg|\lim_{d \uparrow b} \int_a^d dx \, \big[p(x) |f'(x)|^2 + q(x) |f(x)|^2\big]\bigg|   \lb{5.26} \\
& \quad \leq K \|f\|_{L^2((a,b); r(x)dx)} \|\tau f\|_{L^2((a,b); r(x)dx)}, \quad f \in \dom(T_{max}), 
\end{split} 
\end{align}
is satisfied for some $K \in (0,\infty)$ if and only if  
\begin{equation} 
0 < \vartheta_0 < \pi/2   \lb{5.27}
\end{equation}
holds. \\[1mm] 
$(iii)$ If \eqref{5.27} holds, then the best constant $K$ in inequality \eqref{5.26}, denoted by $K_{\dom(T_{max})}((a,b); p,q,r)$, is given by
\begin{equation}
K_{\dom(T_{max})}((a,b); p,q,r) = [\cos(\vartheta_0)]^{-1}.     \lb{5.28} 
\end{equation} 
In particular, $K_{\dom(T_{max})}((a,b); p,q,r) > 1$. \\[1mm] 
$(iv)$ If \eqref{5.27} is satisfied and $K$ is given by $[\cos(\vartheta_0)]^{-1}$ in \eqref{5.28}, then the elements $f \in \dom(T_{max})$ yielding equality in \eqref{5.26} are determined according to the following three mutually exclusive situations: 
\\[1mm]
$(\alpha)$ $f =0$. \\[1mm]
$(\beta)$ There exists $0 \neq f \in \dom(T_{max})$ such that $\tau f =0$ and $f(a) = 0$ or $f^{[1]}(a) =0$  
$($but not both$)$, in which case either side of \eqref{5.26} vanishes. \\[1mm]
$(\gamma)$ $E_+ \cup E_- \neq \emptyset$, $f(x) = C Y_{\pm} (r,x)$, $r \in E_{\pm}$, 
$C \in \bbC \backslash \{0\}$, $x \in [a,b)$.
\end{theorem}

\begin{remark} \lb{r5.3}
$(i)$ Returning to the classical Hardy--Littlewood case (see, \eqref{2.10}, \eqref{2.11} for $p=2$ or 
\eqref{5.14}, \eqref{5.15}, \eqref{5.17}) one obtains
\begin{align}
& m_{\pi/2,\pm}(z) = i z^{-1/2}, \quad \psi_{\pi/2,\pm}(z,x) = - i z^{-1/2} e^{i z^{1/2} x}, \quad z \in \bbC_{\pm}, \; 
x \in (0,\infty),  \no \\
& \vartheta_+ = \pi/3, \quad \vartheta_- = 0, \quad \vartheta_0 = \pi/3, 
\quad K_{H^2((0,\infty))}(\bbR_+;1,0,1) = 2,    \lb{5.29} \\
&E_+ = (0,\infty), \quad E_- = \emptyset,   \no \\ 
& Y_+(\rho,x) 
= e^{-\rho x/2} \sin\big(\big(3^{1/2} \rho x/2\big) - (\pi/3)\big), \quad \rho \in (0,\infty), \; x \in [0,\infty),  \no
\end{align}
confirming \eqref{5.17} once again. 

Reference \cite{EZ78a} treats the analog of Theorem \ref{t5.2} and the Hardy--Littlewood example \eqref{5.29} with the Neumann-type $m$-function $m_{\pi/2,\pm}(\dott)$ replaced by the Dirichlet-type $m$-function $m_{0,\pm}(\dott) = -1/m_{\pi/2,\pm}(\dott)$. \\[1mm]
$(ii)$ Everitt and Zettl \cite{EZ78} showed that the example 
\begin{align}
\begin{split} 
& a = 0, \; b = \infty, \quad p(x) = x^{\beta}, \quad q(x) = 0, \quad r(x) = x^{\alpha}, \quad \alpha > -1, \; 
\beta <1, \\
& \tau = - x^{- \alpha} \frac{d}{dx} x^{\beta} \frac{d}{dx}, \quad x \in (0,\infty), 
\lb{5.30} 
\end{split}
\end{align} 
leads to 
\begin{align}
\begin{split} 
m_{\pi/2,\pm}(z) &= \f{(2+\alpha-\beta)^{2(1-\beta)/(2+\alpha-\beta)}}{1-\beta} 
\f{\Gamma((3+\alpha-2\beta)/(2+\alpha-\beta))}{\Gamma((1+\alpha)/(2+\alpha-\beta))}     \\
& \quad \times e^{i \pi (1-\beta)/(2+\alpha-\beta)} z^{- (1-\beta)/(2+\alpha-\beta)}, \quad z \in \bbC_{\pm},
\end{split}  
\end{align} 
and hence to 
\begin{equation}
K_{\dom(T_{max})}(\bbR_+; x^{\beta},0,x^{\alpha}) = [\cos(\pi (1+\alpha)/(3+2\alpha-\beta))]^{-1}. 
\end{equation}
In particular, Everitt and Zettl \cite{EZ78} (see also \cite{EJ77}) showed that in the special case $\beta = 0$, the constant $K_{\dom(T_{max})}\big(\bbR_+; 1,0,x^{\alpha}\big)$ takes on every value in $(1,\infty)$ as 
$\alpha$ ranges in $(-1, \infty)$. 

For numerous additional explicit examples see, \cite{BBBBDEEKL98}, \cite{BEEK94}, \cite{EE82}, \cite{EE85}, \cite{EE91}, \cite{EEHJ98}, \cite{EEHR86}, \cite{Ev71}, \cite{Ev75}, \cite{Ev93}, \cite{EJ77}, \cite{EK07}; for numerical investigations in this context we refer to \cite{BKP89}. \\[1mm]
$(iii)$ Theorem \ref{t5.2} treats the case where $a$ is a regular endpoint and the endpoint $b$ is in the strong  limit point case. The case where $b$ is either a regular or a limit circle nonoscillatory endpoint is treated in \cite{EE91}. Moreover, the case where $a$ and $b$ are in the strong limit point case (with constant $K=1$ in the analog of \eqref{5.26}, and $\lim_{c \downarrow a, d \uparrow b} \int_c^d dx \, ...$ on the l.h.s. of \eqref{5.26}) is considered in 
\cite{Ev78}. \hfill $\diamond$
\end{remark}

Next, we briefly return to quadratic forms and follow up on some ideas presented by Ph\'ong \cite{Ph81}. 

Let $A$ be a symmetric operator in the Hilbert space $\cH$ bounded from below, that is, 
$A \subseteq A^*$ and for some $c \in \bbR$, $A \ge c I_{\cH}$. We denote by $\ol A$ the closure of $A$ 
in $\cH$, and introduce the associated forms in $\cH$, 
\begin{align}
& q_A(f,g) = (f, A g)_{\cH}, \quad f, g \in \dom(q_A) = \dom(A),   \lb{5.31} \\
& q_{\ol A}(f,g) = (f, {\ol A} g)_{\cH}, \quad f, g \in \dom(q_{\ol A}) = \dom({\ol A}),    \lb{5.32} 
\end{align} 
then the closures of $q_A$ and $q_{\ol A}$ coincide in $\cH$ (cf., e.g., \cite[Lemma~5.1.12]{BHS20}) 
\begin{equation}
\ol{q_A} = \ol{q_{\ol A}}    \lb{5.33}
\end{equation}
and the first representation theorem for forms (see, e.g., \cite[Theorem~4.2.4]{EE18}, 
\cite[Theorem~VI.2.1, Sect.~VI.2.3]{Ka80}) yields 
\begin{equation}
\ol{q_A} (f,g) = (f, A_F g)_{\cH}, \quad f \in \dom (\ol{q_A}), \; g \in \dom(A_F),     \lb{5.34}
\end{equation}
where $A_F \geq c I_{\cH}$ represents the self-adjoint Friedrichs extension of $A$. Due to the fact 
\eqref{5.33}, one infers (cf., e.g., \cite[Lemma~5.3.1]{BHS20})
\begin{equation}
A_F = (\ol A)_F.    \lb{5.35} 
\end{equation} 
The second representation theorem for forms (see, e.g., \cite[Theorem~4.2.8]{EE18}, 
\cite[Theorem~VI.2.123]{Ka80}) then yields the additional result
\begin{align}
\begin{split} 
\ol{q_A} (f,g) = \big((A_F - c I_{\cH})^{1/2} f, (A_F - c I_{\cH})^{1/2} g\big)_{\cH} + c \, \|f\|_{\cH}^2, \\
 f, g \in \dom(\ol{q_A}) = \dom\big(|A_F|^{1/2}\big).&
 \end{split} 
\end{align} 
Moreover, one has the fact (see, e.g., \cite[Theorem~5.3.3]{BHS20}, \cite[Corollary~4.2.7]{EE18})
\begin{equation}
\dom(A_F) = \dom(\ol{q_A}) \cap \dom(A^*) = \dom\big(|A_F|^{1/2}\big) \cap \dom(A^*).   \lb{5.37}
\end{equation}

Returning to \eqref{5.31}, we now consider an extension $Q$ of the form $q_A$ in $\cH$ satisfying 
\begin{equation}
Q(f, g) = (f, Ag)_{\cH}, \quad f \in \dom(Q), \; g \in \dom(A). 
\end{equation} 
Then Cauchy's inequality (cf.\ also Lemma \ref{l4.3}) implies the elementary estimate
\begin{equation}
|Q (f,f)| \leq \|f\|_{\cH} \|A f\|_{\cH}, \quad f \in \dom(A),    \lb{5.39}
\end{equation}
and Ph\'ong \cite[p.~35--36]{Ph81} then points out via the following counterexample that the estimate \eqref{5.39}, in general, permits no extension of the type: There exists a constant $K \in (0,\infty)$ such that 
\begin{equation}
|Q (f,f)| \leq K \, \|f\|_{\cH} \|A^* f\|_{\cH}, \quad f \in \dom(Q) \cap \dom(A^*).    \lb{5.40}
\end{equation}

\begin{example} \lb{e5.4}
Consider 
\begin{align}
\begin{split} 
& \tau_0 = - \f{d^2}{dx^2}, \quad x \in (0,1),   \\
& (T_{0,min} f)(x) = - f''(x), \; x \in (0,1), \quad f \in \dom(T_{0,min}) = H_0^2((0,1)),  \\
& (T_{0,max} f)(x) = - f''(x), \; x \in (0,1), \quad f \in \dom(T_{0,max}) = H^2((0,1)),  \\
& T_{0,max} = T_{0,min}^*, \quad T_{0,max}^* = T_{0,min},  \\
& \ker(T_{0,max}) = \ker(T_{0,min}^*) = {\rm lin.span} \{u_0, u_1\},   \lb{5.41} \\
& \hspace*{1.45cm} u_0(x) = 1, \quad u_1(x) = x, \quad x \in (0,1),  \\ 
& Q_0 (f,g) = (f', g')_{L^2((0,1); dx)}, \quad f, g \in \dom(Q_0) = H^1((0,1)),   \\
& q_{T_{0,min}} (f,g) = (f, (-g''))_{L^2((0,1); dx)}, \quad f, g \in \dom(q_{T_{0,min}}) = H_0^2((0,1)),   \\ 
& \ol{q_{T_{0,min}}} (f,g) = (f', g')_{L^2((0,1); dx)}, \quad f, g \in \dom(\ol{q_{T_{0,min}}}) = H_0^1((0,1)),   \\
& (T_{0,min,F} f)(x) = - f''(x), \; x \in (0,1), \\ 
& f \in \dom(T_{0,min,F}) = H_0^1((0,1)) \cap H^2((0,1)). 
\end{split}
\end{align}
Then 
\begin{align}
\begin{split} 
& Q_0(u_1,u_1) = 1,   \\
& T_{0, min}^* u_1 = T_{0,max} u_1 = 0, 
\end{split} 
\end{align}
and hence there exists no $K \in (0,\infty)$ such that
\begin{align}
\begin{split} 
& |Q_0(f,f)| \leq K \|f\|_{L^2((0,1); dx)} \|f''\|_{L^2((0,1); dx)}, \\
& f \in \dom(Q_0) \cap \dom(T_{0, min}^*) = \dom(T_{0, min}^*) = H^2((0,1)),   \lb{5.43} 
\end{split} 
\end{align}
holds. $($Indeed, taking $f = u_1$ yields $1$ on the l.h.s. of the inequality in \eqref{5.43} and $0$ on its r.h.s.$)$
\end{example} 

Given the negative answer to the problem formulated in connection with \eqref{5.40}, we now describe an elementary affirmative approach involving the Friedrichs extension $A_F$ of $A$.

\begin{lemma} \lb{l5.5}
Suppose $A$ is symmetric and bounded from below in $\cH$. Then, in addition to \eqref{5.39} one has 
\begin{align}
\begin{split} 
|\ol{q_A} (f,f)| \leq \|f\|_{\cH} \|A_F f \|_{\cH} = \|f\|_{\cH} \|A^* f \|_{\cH},&    \lb{5.44} \\
f \in \dom(A_F) = \dom(\ol{q_A}) \cap \dom(A^*).&
\end{split}   
\end{align}
\end{lemma}
\begin{proof}
It suffices to combine \eqref{5.34} and \eqref{5.37}.
\end{proof}

Due to \eqref{5.33}, \eqref{5.35}, one can systematically replace $A$ by $\ol A$ in the above considerations.

\begin{remark} \lb{r5.6}
$(i)$ Inequality \eqref{5.39} applied to $T_{0,min}$ in Example \ref{e5.4} yields 
\begin{equation}
\|f'\|_{L^2((0,1); dx)}^2 \leq \|f\|_{L^2((0,1); dx)} \|f''\|_{L^2((0,1); dx)}, 
\quad f \in H_0^2((0,1)),     \lb{5.45} 
\end{equation}
and applying Lemma \ref{l5.5}  to $T_{0,min}$ in Example \ref{e5.4} then implies the improvement
\begin{equation}
\|f'\|_{L^2((0,1); dx)}^2 \leq \|f\|_{L^2((0,1); dx)} \|f''\|_{L^2((0,1); dx)}, 
\quad f \in H_0^1((0,1)) \cap H^2((0,1)).      \lb{5.46} 
\end{equation}
$(ii)$ Similarly, Lemma \ref{l5.5} applied in the context of \eqref{5.14}--\eqref{5.20} yields the following improvement of \eqref{5.16} and \eqref{5.18} in the form 
\begin{equation}
\|f'\|_{L^2((0,\infty); dx)}^2 \leq \|f\|_{L^2((0,\infty); dx)} \|f''\|_{L^2((0,\infty); dx)}, 
\quad f \in H_0^1((0,1)) \cap H^2((0,\infty)),     \lb{5.47} 
\end{equation}
that is,
\begin{equation} 
K_{H_0^1((0,\infty)) \cap H^2((0,\infty))} = 1. 
\end{equation}
Indeed, this follows from the fact (cf.\ \eqref{5.20}) 
\begin{align}
& \big(\big(- A_0^2\big)^* f\big)(x) = - f''(x), \; x \in (0,\infty), \quad  
f \in \dom\big(\big(- A_0^2\big)^*\big) = H^2((0,\infty)),    \no \\  
& \big(\big(- A_0^2\big)_F f\big)(x) = - f''(x), \; x \in (0,\infty), \\
& f \in \dom\big(\big(- A_0^2\big)_F\big) = H_0^1((0,\infty)) \cap H^2((0,\infty)).    \no 
\end{align}
${}$ \hfill $\diamond$
\end{remark}

We will provide one more explicitly solvable example illustrating Lemma \ref{l5.5}, but due to the complexity of the example we will present it separately in the next section.

\section{An Explicitly Solvable Example: A Generalized Bessel Equation} \lb{s6} 

In this section we analyze the following explicitly solvable example detailed in \eqref{6.1}, \eqref{6.2} below. 

Let $a=0$, $b = \infty$ in \eqref{5.4}, and consider the concrete example 
\begin{equation}
\begin{split} \lb{6.1}
p(x)=x^\b ,\quad r(x)=x^\a, \quad q (x) = \frac{(2+\a-\b)^2\g^2-(1-\b)^2}{4}x^{\b-2}, \\
\a>-1,\ \b<1,\ \g\geq0,\ x\in(0,\infty).     
\end{split}
\end{equation}
Then 
\begin{align}
\begin{split}
\tau_{\a,\b,\g} = x^{-\a}\left[-\frac{d}{dx}x^\b\frac{d}{dx} +\frac{(2+\a-\b)^2\g^2-(1-\b)^2}{4}x^{\b-2}\right],\\
\a>-1,\ \b<1,\ \g\geq0,\ x\in(0,\infty),     \lb{6.2} 
\end{split}
\end{align}
is singular at the endpoint $0$ (since the potential, $q$ is not integrable near $x=0$) and in the limit point case at $\infty$. Furthermore, $\tau_{\a,\b,\g}$ is in the limit circle case at $x=0$ if $0\leq\g<1$ and in the limit point case at $x=0$ when $\g\geq1$. 

Solutions to \eqref{6.1} are given by (cf.\ \cite[No.~2.162, p.~440]{Ka61})
\begin{align}
y_{1,\a,\b,\g}(z,x)&=x^{(1-\b)/2} J_{\gamma}\big(2z^{1/2} x^{(2+\a-\b)/2}/(2+\a-\b)\big),\quad \g\geq0,\\
y_{2,\a,\b,\g}(z,x)&=\begin{cases}
x^{(1-\b)/2} J_{-\gamma}\big(2z^{1/2} x^{(2+\a-\b)/2}/(2+\a-\b)\big), & \g\notin\bbN_0,\\
x^{(1-\b)/2} Y_{\g}\big(2z^{1/2} x^{(2+\a-\b)/2}/(2+\a-\b)\big), & \g\in\bbN_0,
\end{cases}\ \g\geq0,
\end{align}
where $J_{\nu}(\dott), Y_{\nu}(\dott)$ are the standard Bessel functions of order $\nu \in \bbR$ 
(cf.\ \cite[Ch.~9]{AS72}).

In the following we assume that 
\begin{equation}
\gamma \in [0,1) 
\end{equation}
to ensure the limit circle case at $x=0$. In this case it suffices to focus on  the generalized boundary values at the singular endpoint $x = 0$ following \cite{GLN20}. For this purpose we introduce principal and nonprincipal solutions $u_{0,\a,\b,\gamma}(0, \dott)$ and $\hatt u_{0,\a,\b,\gamma}(0, \dott)$ of $\tau_{\alpha,\beta,\gamma} u = 0$ at $x=0$ by
\begin{align}
\begin{split} 
u_{0,\a,\b,\gamma}(0, x) &= (1-\b)^{-1}x^{[1-\b+(2+\a-\b)\g]/2}, \quad \gamma \in [0,1),   \\
\hatt u_{0,\a,\b,\gamma}(0, x) &= \begin{cases} (1-\b)[(2+\a-\b) \gamma]^{-1} x^{[1-\b-(2+\a-\b)\g]/2}, & \gamma \in (0,1),     \lb{6.6} \\
(1-\b)x^{(1-\b)/2} \ln(1/x), & \gamma =0,  \end{cases}\\
&\hspace*{4.5cm} \a>-1,\; \b<1,\; x \in (0,1).
\end{split} 
\end{align}

\begin{remark}
Since the singularity of $q$ at $x=0$ renders $\tau_{\alpha,\beta,\gamma}$ singular at $x=0$ (unless, of 
course, $\gamma = (1-\beta)/(2+\alpha-\beta)$, in which case $\tau_{\alpha,\beta,(1-\beta)/(2+\alpha-\beta)}$ is regular at $x=0$), there is a certain freedom in the choice of the multiplicative constant in the principal solution $u_{0,\alpha,\beta,\gamma}$ of $\tau_{\alpha,\beta,\gamma} u = 0$ at $x=0$. Our choice of $(1-\beta)^{-1}$ in \eqref{6.6} reflects continuity in the parameters when comparing to boundary conditions in the regular case (cf.\ \cite[Remark~3.12\,$(ii)$]{GLN20}), that is, in the case $\alpha > -1$, $\beta < 1$, and 
$\gamma = (1-\beta)/(2+\alpha-\beta)$ treated in \cite{EZ78} (see Remark \ref{r5.3}\,$(ii)$). 
\hfill $\diamond$ 
\end{remark}

According to \cite{GLN20} the generalized boundary values for $g \in \dom(T_{\alpha, \beta, \gamma,max})$ are then of the form
\begin{align}
\begin{split} 
\wti g(0) &= - W(u_{0,\alpha,\beta,\gamma}(0, \dott), g)(0)   \\
&= \begin{cases} \lim_{x \downarrow 0} g(x)\big/\big[(1-\b)[(2+\a-\b) \gamma]^{-1} x^{[1-\b-(2+\a-\b)\g]/2}\big], & 
\gamma \in (0,1), \\[1mm]
\lim_{x \downarrow 0} g(x)\big/\big[(1-\b)x^{(1-\b)/2} \ln(1/x)\big], & \gamma =0, 
\end{cases} 
\end{split} \\
\begin{split} 
\wti g^{\, \prime} (0) &= W(\hatt u_{0,\alpha,\beta,\gamma}(0, \dott), g)(0)   \\
&= \begin{cases} \lim_{x \downarrow 0} \big[g(x) - \wti g(0) (1-\b)[(2+\a-\b) \gamma]^{-1} x^{[1-\b-(2+\a-\b)\g]/2}\big]\\
\qquad\qquad\big/\big[(1-\b)^{-1}x^{[1-\b+(2+\a-\b)\g]/2}\big], 
& \hspace{-.2cm}\gamma \in (0,1), \\[1mm]
\lim_{x \downarrow 0} \big[g(x) - \wti g(0) (1-\b)x^{(1-\b)/2} \ln(1/x)\big]\\
\qquad\qquad\big/\big[(1-\b)^{-1}x^{(1-\b)/2}\big], & \hspace{-.2cm} \gamma =0.
\end{cases}
\end{split} 
\end{align}

Next, introducing the standard normalized fundamental system of solutions 
$\phi_{\alpha, \beta, \gamma, 0}(z, \dott), \theta_{\alpha, \beta, \gamma, 0}(z, \dott)$ of 
$\tau_{\alpha, \beta, \gamma} u = z u$, $z \in \bbC$, that is real-valued for $z \in \bbR$ and entire with respect to $z \in \bbC$ by 
\begin{align}
\begin{split} 
\wti \phi_{\alpha, \beta, \gamma, 0} (z,0) &= 0, \quad 
\wti \phi_{\alpha, \beta, \gamma, 0}^{\, \prime} (z,0) = 1,    \\ 
\wti \theta_{\alpha, \beta, \gamma, 0} (z,0) &= 1, \quad \, 
\wti \theta_{\alpha, \beta, \gamma, 0}^{\, \prime} (z,0) = 0, \quad 
z \in \bbC,     \lb{6.8}
\end{split} 
\end{align}
one obtains explicitly, 
\begin{align}
& \phi_{\a,\b,\gamma,0}(z,x) = \begin{cases} (1-\b)^{-1}(2+\a-\b)^\gamma\Gamma(1+\g) z^{- \gamma/2} 
y_{1,\a,\b,\g}(z,x), & \gamma \in (0,1), \\[1mm]
(1-\b)^{-1}y_{1,\a,\b,0}(z,x), & \gamma = 0, 
\end{cases}    \no \\
& \hspace*{8.2cm} z \in \bbC, \; x \in (0,\infty),      \\
& \theta_{\a,\b,\gamma,0}(z,x) = \begin{cases} (1-\b)(2+\a-\b)^{-\gamma - 1} \gamma^{-1} \Gamma(1 - \gamma) 
z^{\gamma/2} y_{2,\a,\b,\g}(z,x), & \hspace{-.1cm} \gamma \in (0,1), \\[1mm]
(1-\b)(2+\a-\b)^{-1} [- \pi y_{2,\a,\b,0}(z,x)\\
\quad+(\ln(z)-2\ln(2+\a-\b)+2\g_E) y_{1,\a,\b,0}(z,x)], & \hspace{-.1cm} \gamma =0, 
\end{cases}      \no \\ 
& \hspace*{8.2cm} z \in \bbC, \; x \in (0,\infty),    \\
& W(\theta_{\a,\b,\g,0}(z,\dott), \phi_{\a,\b,\gamma,0}(z,\dott)) =1, \quad z \in \bbC,
\end{align}
where $\Gamma(\dott)$ denotes the Gamma function, and $\gamma_{E} = 0.57721\dots$ represents Euler's constant.

Since $\tau_{\a,\b,\gamma}$ is in the limit point case at $\infty$ (actually, it is in the strong limit point case at infinity since $q$ is bounded on any interval of the form $[R,\infty)$, $R>0$, and the strong limit point property of $\tau_{\alpha, \beta, \gamma =(1-\beta)/(2+\alpha-\beta)}$ has been shown in \cite{EZ78}), in order to find the $m$-function corresponding to the Friedrichs (resp., Dirichlet) boundary condition at $x=0$, one considers the requirement
\begin{align}
& \psi_{\a,\b,\g,0}(z,\dott)=\theta_{\a,\b,\gamma,0}(z,\dott) 
+ m_{\a,\b,\g,0}(z)\phi_{\a,\b,\gamma,0}(z,\dott)\in L^2((0,\infty);x^\a dx),    \no \\
& \hspace*{9.5cm} z \in \bbC \backslash \bbR. 
\end{align}
This implies
\begin{align}
& \psi_{\a,\b,\gamma,0}(z,x) = \begin{cases} i (1-\b)(2+\a-\b)^{-\gamma -1} \gamma^{-1} \Gamma(1 - \gamma) \sin(\pi \gamma) 
z^{\gamma/2}  \\
\quad \times x^{(1-\b)/2} H^{(1)}_{\gamma}\big(2z^{1/2} x^{(2+\a-\b)/2}/(2+\a-\b)\big), & \gamma \in (0,1), \\[1mm]
i \pi(1-\b)/(2+\a-\b) x^{(1-\b)/2}\\
\quad \times H^{(1)}_0\big(2z^{1/2} x^{(2+\a-\b)/2}/(2+\a-\b)\big), & \gamma = 0, 
\end{cases}    \no \\
& \hspace*{7cm} z \in \bbC \backslash [0,\infty), \; x \in (0,\infty), \label{6.14}   \\
& m_{\a,\b,\g,0}(z) = \begin{cases} - e^{-i \pi \gamma} (1-\b)^2(2+\a-\b)^{- 2 \gamma - 1} \gamma^{-1} \\
\quad \times [\Gamma(1 - \gamma)/\Gamma(1+\gamma)] z^{\gamma}, & \gamma \in (0,1), \\[1mm]
 (1-\b)^2/(2+\a-\b)   \\
\quad \times [i \pi -\ln(z)+ 2\ln(2+\a-\b)- 2\gamma_{E}], & \gamma = 0,    
\end{cases}     \label{6.15} \\
& \hspace*{7.8cm} z \in \bbC \backslash [0,\infty),    \no 
\end{align}
where $H_{\nu}^{(1)}(\dott)$ is the Hankel function of the first kind and of order $\nu \in \bbR$ 
(cf.\ \cite[Ch.~9]{AS72}). In particular, the results \eqref{6.14} and \eqref{6.15} coincide with the ones obtained in \cite{GLN20} when $\a=\b=0$ and \cite{EZ78} when $\g=(1-\b)/(2+\a-\b)$.

$L^2$-realizations associated with the differential expression $\tau_{\alpha,\beta,\gamma}$ are then introduced as usual by
\begin{align}
& (T_{\alpha,\beta,\gamma, max} f)(x) = (\tau_{\alpha,\beta,\gamma} f)(x), \quad x \in (0,\infty),   \no \\
& \, f \in \dom(T_{\alpha,\beta,\gamma, max}) = \big\{g \in L^2((0,\infty); x^{\alpha} dx) \, \big | \, 
g, g' \in AC_{loc}((0,\infty)); \\  
& \hspace*{6.4cm} \tau_{\alpha,\beta,\gamma} g \in L^2((0,\infty); x^{\alpha} dx) \big\},    \no \\
\begin{split} 
& (T_{\alpha,\beta,\gamma, min} f)(x) = (\tau_{\alpha,\beta,\gamma} f)(x), \quad x \in (0,\infty),    \\
& \, f \in \dom(T_{\alpha,\beta,\gamma, min}) = \big\{g \in \dom(T_{\alpha,\beta,\gamma, max}) \, \big | \, 
\wti g (0) =0, \, \wti g^{\, \prime} (0) =0 \big\},    \\
\end{split} 
\end{align}
in particular,
\begin{equation}
T_{\alpha,\beta,\gamma, min}^* = T_{\alpha,\beta,\gamma, max}, \quad 
T_{\alpha,\beta,\gamma, max}^* = T_{\alpha,\beta,\gamma, min}.
\end{equation}
Thus, following Kalf \cite{Ka78} (see also \cite{NZ92}, \cite{Ro85}) one obtains for the Friedrichs extension 
$T_{\alpha,\beta,\gamma, F}$ of $T_{\alpha,\beta,\gamma, min}$,
\begin{align}
\begin{split} 
& (T_{\alpha,\beta,\gamma, F} f)(x) = (\tau_{\alpha,\beta,\gamma} f)(x), \quad x \in (0,\infty),    \\
& \, f \in \dom(T_{\alpha,\beta,\gamma, F}) = \big\{g \in \dom(T_{\alpha,\beta,\gamma, max}) \, \big | \, 
\wti g (0) =0 \big\}.   \lb{6.19} \\
\end{split} 
\end{align}
Since 
\begin{equation} 
u_{0,\a,\b,\gamma}(0, x) = (1-\b)^{-1}x^{[1-\b+(2+\a-\b)\g]/2} > 0, \quad \gamma \in [0,1), \; x \in (0,1),
\end{equation} 
standard oscillation theory implies that 
\begin{equation}
T_{\alpha,\beta,\gamma, min} \geq 0 \, \text{ and hence } \, T_{\alpha,\beta,\gamma, F} \geq 0.
\end{equation}
Incidentally, we note that $T_{\alpha,\beta,\gamma, F} \geq 0$ also follows from the Stieltjes inversion 
formula applied to $m_{\a,\b,\g,0}(\dott)$, proving that the spectral measure corresponding to 
$T_{\alpha,\beta,\gamma, F}$ (i.e., the measure in the Nevanlinna--Herglotz representation of 
$m_{\a,\b,\g,0}(\dott)$) has no support in $(-\infty,0)$.

Next one observes that  
\begin{align}
& u_{0,\alpha,\beta,\gamma}(x) [x^{\beta} u_{0,\alpha,\beta,\gamma}'(x)] \underset{x \downarrow 0}{=} 
\Oh\big(x^{(2 + \alpha - \beta)\gamma}\big), \quad \gamma \in [0,1),    \lb{6.22} \\
& u_{0,\alpha,\beta,\gamma}(x)/ \hatt u_{0,\alpha,\beta,\gamma}(x) 
\underset{x \downarrow 0}{=} 
\begin{cases} 
\Oh\big(x^{(2 + \alpha - \beta)\gamma}\big), & \gamma \in (0,1), \\
\Oh\big([\ln(1/x)]^{-1}\big), & \gamma = 0,
\end{cases}    \lb{6.23} 
\end{align}
in particular, $u_{0,\alpha,\beta,\gamma}(x) \big[x^{\beta} u_{0,\alpha,\beta,\gamma}'(x)\big] \underset{x \downarrow 0}{=} \Oh\big(u_{0,\alpha,\beta,\gamma}(x)/ \hatt u_{0,\alpha,\beta,\gamma}(x)\big)$ for $\gamma \in (0,1)$ (but not when $\gamma = 0$), a condition isolated in \cite{Hi74}, and that 
\begin{align}
\begin{split} 
& \hatt u_{0,\alpha,\beta,\gamma}(x) 
\big[x^{\beta} \hatt u_{0,\alpha,\beta,\gamma}'(x)\big] \underset{x \downarrow 0}{=} \begin{cases} 
\Oh\big(x^{- (2 - \alpha + \beta)\gamma}\big), & \gamma \in (0,1), \\
\Oh\big([\ln(x)]^2\big), & \gamma = 0,  \end{cases}    \\ 
& \hspace*{4.6cm} \text{does not exist}.    \lb{6.24} 
\end{split} 
\end{align}

Turning our attention to $x = \infty$, we now introduce (non-normalized) principal and nonprincipal solutions 
$u_{\infty,\a,\b,\gamma}(0, \dott)$ and $\hatt u_{\infty,\a,\b,\gamma}(0, \dott)$ of 
$\tau_{\alpha,\beta,\gamma} u = 0$ at $x=\infty$ by
\begin{align}
\begin{split} 
u_{\infty,\a,\b,\gamma}(0, x) &= x^{[1-\b-(2+\a-\b)\g]/2}, \quad \gamma \in [0,1),   \\
\hatt u_{\infty,\a,\b,\gamma}(0, x) &= \begin{cases} x^{[1-\b+(2+\a-\b)\g]/2}, & \gamma \in (0,1), \\
x^{(1-\b)/2} \ln(x), & \gamma =0,  \end{cases}\\
&\hspace*{5.5mm} \a>-1,\; \b<1,\; x \in (1, \infty),     \lb{6.25} 
\end{split} 
\end{align}
to obtain 
\begin{align}
& u_{\infty,\alpha,\beta,\gamma}(x) \big[x^{\beta} u_{\infty,\alpha,\beta,\gamma}'(x)\big] 
\underset{x \uparrow \infty}{=} 
\Oh\big(x^{-(2 + \alpha - \beta)\gamma}\big), \quad \gamma \in [0,1),     \lb{6.26} \\
& u_{\infty,\alpha,\beta,\gamma}(x)/ \hatt u_{\infty,\alpha,\beta,\gamma}(x) 
\underset{x \uparrow \infty}{=} 
\begin{cases} 
\Oh\big(x^{-(2 + \alpha - \beta)\gamma}\big), &\gamma \in (0,1), \\
\Oh\big([\ln(x)]^{-1}\big), & \gamma = 0,
\end{cases}     \lb{6.27} 
\end{align}
hence, one infers once again that 
\begin{equation}
u_{\infty,\alpha,\beta,\gamma}(x) \big[x^{\beta} u_{\infty,\alpha,\beta,\gamma}'(x)\big] 
\underset{x \uparrow \infty}{=} 
\Oh\big(u_{\infty,\alpha,\beta,\gamma}(x)/ \hatt u_{\infty,\alpha,\beta,\gamma}(x)\big) 
\, \text{ for $\gamma \in (0,1)$} 
\end{equation} 
(but not when $\gamma = 0$), and that 
\begin{align}
\begin{split} 
& \hatt u_{\infty,\alpha,\beta,\gamma}(x) 
\big[x^{\beta} \hatt u_{\infty,\alpha,\beta,\gamma}'(x)\big] \underset{x \uparrow \infty}{=} \begin{cases} 
\Oh\big(x^{ (2 - \alpha + \beta)\gamma}\big), & \gamma \in (0,1), \\
\Oh\big([\ln(x)]^2\big), & \gamma = 0,  \end{cases}    \\ 
& \hspace*{5cm} \text{does not exist}.     \lb{6.28} 
\end{split} 
\end{align}

Given \eqref{6.22}--\eqref{6.28}, \cite[Corollary~3]{Ro85} applies near $x=0$ as well as near $x = \infty$, and thus one obtains in addition to \eqref{6.19} that 
\begin{align}
& (T_{\alpha,\beta,\gamma, F} f)(x) = (\tau_{\alpha,\beta,\gamma} f)(x), \quad x \in (0,\infty),    \no \\
& f \in \dom(T_{\alpha,\beta,\gamma, F}) = \big\{g \in \dom(T_{\alpha,\beta,\gamma, max}) \, \big | \, 
g' \in L^2((0,\infty); x^{\beta} dx);       \lb{6.29} \\ 
& \hspace*{3.45cm} g \in L^2((0,\infty); x^{\beta - 2} dx)\big\}, \quad \alpha > -1, \; \beta < 1, \; \gamma \in (0,1).  \no
\end{align}
Here we used the Monotone Convergence Theorem to conclude that 
\begin{align}
\begin{split} 
& \infty > \lim_{c\downarrow 0, \, d \uparrow \infty} \int_c^d dx \, x^{\beta - 2} |g(x)|^2  
= \lim_{c\downarrow 0, \, d \uparrow \infty} \int_0^{\infty} dx \, x^{\beta - 2} \chi_{[c,d]}(x) |g(x)|^2  \\
& \quad = \int_0^{\infty} dx \, x^{\beta - 2} |g(x)|^2, \quad g \in \dom(T_{\alpha,\beta,\gamma, max}). 
\end{split} 
\end{align}

One notes that the characterization \eqref{6.29} cannot hold for $\gamma = 0$ by simply choosing $g$ to equal $u_{0,\alpha,\beta,0}$ near $x=0$ and $u_{\infty,\alpha,\beta,0}$ near $x=\infty$ (all integrands then are of the form $\Oh(1/x)$ near $x =0, \infty$). However, for $\gamma > 0$ one can improve \eqref{6.29} with the help of a weighted Hardy inequality as follows.

\begin{lemma} \lb{l6.2} $($Kalf and Walter \cite[Lemma~1\,(a)]{KW72}$)$.  
Suppose that 
\begin{align} 
\begin{split} 
& \text{$0 < p$~a.e.~on $(0,\infty)$, $p^{-1} \in L^1((0,c); dx)$ for all $c \in (0,\infty)$},    \\
& f \in AC_{loc}((0,\infty)), \quad f' \in L^2((0,\infty); p(x) dx), \quad \liminf_{x \downarrow 0} |f(x)| = 0. 
\end{split} 
\end{align} 
Then 
\begin{equation}
\lim_{x \downarrow 0} \f{|u(x)|^2}{\int_0^x dt \, p(t)^{-1}} = 0,
\end{equation}
and for all $R \in (0,\infty) \cup \{\infty\}$, 
\begin{equation}
\int_0^R dx \, p(x) |f'(x)|^2 \geq \f{1}{4} \int_0^R dx \, \f{|f(x)|^2}{p(x) \Big[\int_0^x dt \, p(t)^{-1}\Big]^2}.
\lb{6.34} 
\end{equation}
\end{lemma} 

The generalized version of weighted Hardy inequalities, especially, its integral inequality version  
(replacing $f(x)$ by $\int_a^x dt \, F(t)$ or $\int_x^b dt \, F(t)$, and hence $f'(x)$ by $F(x)$, etc.) in
$L^{2}((a,b);dx)$ was established by Talenti \cite{Ta69} and Tomaselli \cite{To69} in 1969 and independently rediscovered by Chisholm and Everitt \cite{CE70/71} in 1971 (see also \cite{CEL99} for a more general result in the conjugate index case $1/p+1/q=1$, and \cite{GLMW18}, \cite{GLMP20}, and the references therein, 
for recent developments).  In addition, a 1972 paper by Muckenhoupt \cite{Mu72}
has further generalizations. For additional information on the fascinating history of this type of inequalities we refer to the excellent account in the \cite[Ch.~4, pp.~33--37]{KMP07}. 

An application of Lemma \ref{l6.2} to the example at hand with $p(x) = x^{\beta}$, $\beta < 1$, thus yields the special power weighted Hardy inequality
\begin{equation}
\int_0^R dx \, x^{\beta} |f'(x)|^2 \geq \f{(1-\beta)^2}{4} \int_0^{R} dx \, x^{\beta - 2} |f(x)|^2, 
\quad \beta < 1, \; R \in (0,\infty) \cup \{\infty\},    \lb{6.35}
\end{equation}
assuming 
\begin{equation}
f \in AC_{loc}((0,\infty)), \quad f' \in L^2\big((0,\infty); x^{\beta} dx\big), \quad \liminf_{x \downarrow 0} |f(x)| = 0. 
\lb{6.36} 
\end{equation}
Since $q$ is of the form, $q(x) = C_{\alpha,\beta,\gamma} x^{\beta - 2}$, a combination of \eqref{6.29} and \eqref{6.35}, \eqref{6.36} yields 
\begin{align}
& (T_{\alpha,\beta,\gamma, F} f)(x) = (\tau_{\alpha,\beta,\gamma} f)(x), \quad x \in (0,\infty),    \no \\
& f \in \dom(T_{\alpha,\beta,\gamma, F}) = \big\{g \in \dom(T_{\alpha,\beta,\gamma, max}) \, \big | \, 
g' \in L^2((0,\infty); x^{\beta}dx);       \lb{6.37} \\ 
& \hspace*{3.4cm} \liminf_{x \downarrow 0} |g(x)| = 0\big\}, \quad 
\alpha > -1, \; \beta < 1, \; \gamma \in (0,1).    \no
\end{align}

Finally, combining Lemma \ref{l5.5} and \eqref{6.37} yields 
\begin{align}
& \int_0^{\infty} dx \, \Big[x^{\beta} |f'(x)|^2 + 4^{-1}\big[(2 + \alpha - \beta)^2 \gamma^2 - (1-\beta)^2\big] 
x^{\beta - 2} |f(x)|^2\Big]    \no \\
& \quad \leq \bigg(\int_0^{\infty} x^{\alpha}dx \, |f(x)|^2\bigg)^{1/2}
\bigg(\int_0^{\infty} x^{\alpha}dx \, |(\tau_{\alpha,\beta,\gamma} f)(x)|^2\bigg)^{1/2},   \lb{6.38} \\
& \hspace*{2.2cm}  f \in \dom(T_{\alpha,\beta,\gamma,F}), \;  \alpha > -1, \; \beta < 1, \; \gamma \in (0,1),    \no  
\end{align}
extending the special case $\alpha=\beta=0$ treated at the end of \cite{EK07}.

Judging from Theorem \ref{t5.2} in the regular context at $x=a$, one might naively guess at first that an inequality of the kind of \eqref{6.38} might extend to other self-adjoint extensions of 
$T_{\alpha,\beta,\gamma,min}$ (perhaps, even to $T_{\alpha,\beta,\gamma,max}$). In the remainder of this section we will prove that this fails and that \eqref{6.38} cannot extend to any self-adjoint extension of 
$T_{\alpha,\beta,\gamma, min}$ other than the Friedrichs extension $T_{\alpha,\beta,\gamma,F}$. 

We start by noting (cf.\ \cite[Sects.~3, 4]{GLN20}) all self-adjoint extensions of $T_{\alpha,\beta,\gamma,min}$
are of the form,
\begin{align}
\begin{split} 
& (T_{\alpha,\beta,\gamma,\delta} f)(x) = (\tau_{\alpha,\beta,\gamma} f)(x), \quad x \in (0,\infty), \; 
\delta \in [0,\pi),    \\
& \, f \in \dom(T_{\alpha,\beta,\gamma,\delta}) = \big\{g \in \dom(T_{\alpha,\beta,\gamma,max}) \, \big | \, 
\sin(\delta) \wti g^{\, \prime} (0) + \cos(\delta) \wti g (0) = 0\big\},   \lb{6.39} \\
\end{split} 
\end{align}
and hence the Friedrichs extension $T_{\alpha,\beta,\gamma,F}$ of $T_{\alpha,\beta,\gamma,min}$ corresponds to the case $\delta=0$. In order to describe the resolvent of $T_{\alpha,\beta,\gamma,\delta}$ we need a few preparations. Introducing,
\begin{align}
& \phi_{\alpha,\beta,\gamma,\delta}(z,x) = \cos(\delta) \phi_{\alpha,\beta,\gamma,0}(z,x) - \sin(\delta) 
\theta_{\alpha,\beta,\gamma,0}(z,x),    \no \\
& \theta_{\alpha,\beta,\gamma,\delta}(z,x) = \sin(\delta) \phi_{\alpha,\beta,\gamma,0}(z,x) + \cos(\delta) 
\theta_{\alpha,\beta,\gamma,0}(z,x),    \lb{6.40} \\
& \hspace*{4.1cm} z \in \bbC, \; x \in (0, \infty), \; \delta \in [0,\pi),    \no 
\end{align}
observing
\begin{equation}
W(\theta_{\alpha,\beta,\gamma,\delta}(z,\dott), \phi_{\alpha,\beta,\gamma,\delta}(z,\dott)) = 1, \quad 
z \in \bbC,
\end{equation}
one considers the Weyl--Titchmarsh solutions 
\begin{align}
& \psi_{\a,\b,\g,\delta}(z,\dott)=\theta_{\a,\b,\gamma,\delta}(z,\dott) 
+ m_{\a,\b,\g,\delta}(z)\phi_{\a,\b,\gamma,\delta}(z,\dott)\in L^2((0,\infty);x^\a dx),    \no \\
& \hspace*{8cm} z \in \bbC \backslash \bbR, \; \; \delta \in [0,\pi),
\end{align}
where
\begin{equation}
m_{\a,\b,\g,\delta}(z) = \f{- \sin(\delta) + \cos(\delta) m_{\a,\b,\g,0}(z)}{\cos(\delta) 
+ \sin(\delta) m_{\a,\b,\g,0}(z)}, \quad z \in \bbC \backslash \bbR, \; \delta \in [0,\pi). 
\end{equation}
Since $\tau_{\alpha,\beta,\gamma}$ is in the limit point case at $\infty$, one concludes that
\begin{equation}
\psi_{\a,\b,\g,\delta}(z,\dott) = C_{\a,\b,\g,\delta}(z) \psi_{\a,\b,\g,0}(z,\dott) 
\end{equation}
for some constant $C_{\a,\b,\g,\delta}(z) \in \bbC \backslash \{0\}$. 

One then obtains for the Green's function $G_{\alpha,\beta,\gamma,\delta} (z, \dott,\dott)$ of 
$T_{\alpha,\beta,\gamma,\delta}$ (i.e., the integral kernel of the resolvent 
$(T_{\alpha,\beta,\gamma,\delta} - z I_{L^2((0,\infty); x^{\alpha} dx)})^{-1}$), 
\begin{equation}
G_{\alpha,\beta,\gamma,\delta} (z, x, x') = \begin{cases}
\phi_{\alpha,\beta,\gamma,\delta}(z,x) \psi_{\a,\b,\g,\delta}(z,x'), & 0 < x \leq x' < \infty, \\
\phi_{\alpha,\beta,\gamma,\delta}(z,x') \psi_{\a,\b,\g,\delta}(z,x), & 0 < x' \leq x < \infty,
\end{cases} \quad z \in \bbC \backslash \bbR.      \lb{6.44} 
\end{equation}

Next, consider $g \in L^2((0,\infty); x^{\alpha} dx)$, with $\supp\,(g) \subset (0,\infty)$ compact (i.e., the support of $g$ is away from $0$ and from $\infty$), then 
\begin{equation}
f(z,\delta,\dott) = (T_{\alpha,\beta,\gamma,\delta} - z I_{L^2((0,\infty); x^{\alpha} dx)})^{-1} g \in 
\dom(T_{\alpha,\beta,\gamma,\delta}), \quad z \in \bbC \backslash \bbR,  
\end{equation} 
and the set $\cD_{\alpha,\beta,\gamma,\delta,z}$ of such $f(z,\delta,\dott)$ forms an operator core of  
$T_{\alpha,\beta,\gamma,\delta}$, that is, 
$\ol{T_{\alpha,\beta,\gamma,\delta}\big|_{\cD_{\alpha,\beta,\gamma,\delta,z}}} = T_{\alpha,\beta,\gamma,\delta}$ (since $T_{\alpha,\beta,\gamma,min}$ is symmetric and the set of 
$g \in L^2((0,\infty); x^{\alpha} dx)$ with $\supp(g) \subset (0,\infty)$ compact is dense in 
$ L^2((0,\infty); x^{\alpha} dx)$, see \cite[Corollary, p.~257]{RS75}). 

One then obtains from \eqref{6.44}, 
\begin{align}
f(z,\delta,x) &=  \psi_{\a,\b,\g,\delta}(z,x) \int_0^x (x')^{\alpha} dx' \, \phi_{\alpha,\beta,\gamma,\delta}(z,x') g(x') 
\no \\
& \quad + \phi_{\alpha,\beta,\gamma,\delta}(z,x) \int_x^{\infty} (x')^{\alpha} dx' \, \psi_{\alpha,\beta,\gamma,\delta}(z,x') g(x')   \no \\
&= \phi_{\alpha,\beta,\gamma,\delta}(z,x) \int_0^{\infty} (x')^{\alpha} dx' \, \psi_{\alpha,\beta,\gamma,\delta}(z,x') g(x'), \quad z \in \bbC \backslash \bbR,  
\end{align}
for $0 < x < \inf(\supp\,(g))$. Changing $z \in \bbC \backslash \bbR$ a bit if necessary, we may assume without loss of generality that 
\begin{equation}
\int_0^{\infty} (x')^{\alpha} dx' \, \psi_{\alpha,\beta,\gamma,\delta}(z,x') g(x') \neq 0. 
\end{equation}
Thus, for some constant $c_{\alpha,\beta,\gamma}(z) \in \bbC\backslash\{0\}$, 
\begin{equation}
\phi_{\alpha,\beta,\gamma,\delta}(z,x) \underset{x \downarrow 0}{=} \sin(\delta) c_{\alpha,\beta,\gamma}(z)  
x^{[(1-\beta) - (2 + \alpha - \beta)\gamma]/2}, \quad \delta \in (0,\pi), 
\end{equation}
and hence for all $\alpha > -1$, $\beta < 1$, $\gamma \in (0,1)$, and $\delta \in (0,\pi)$,
\begin{equation}
\lim_{\varepsilon \downarrow 0} \int_{\varepsilon}^{\infty} dx \, x^{\beta} |f'(z,\delta,x)|^2 =\infty,\quad  \lim_{\varepsilon \downarrow 0} \int_{\varepsilon}^{\infty} dx \, x^{\beta - 2} |f(z,\delta,x)|^2 = \infty,
\end{equation}
provided $(2+\alpha-\beta)\gamma-(1-\beta)\neq 0$.  Thus, as long as $q\not \equiv 0$, no cancellation in the analog of the left-hand side of \eqref{6.38}, namely, 
\begin{equation}
\int_0^{\infty} dx \, \Big[x^{\beta} |f'(z,\delta,x)|^2 + 4^{-1}\big[(2 + \alpha - \beta)^2 \gamma^2 - (1-\beta)^2\big] 
x^{\beta - 2} |f(z,\delta,x)|^2\Big], 
\end{equation}
can possibly occur. Hence, for any $\delta \in (0,\pi)$ (i.e., for all self-adjoint extensions of  
$T_{\alpha,\beta,\gamma,min}$ other than the Friedrichs extension $T_{\alpha,\beta,\gamma,F}$), there exists a core $\cD_{\alpha,\beta,\gamma,\delta,z}$ for $T_{\alpha,\beta,\gamma,\delta}$ such that the analog of the
 left-hand side of \eqref{6.38} diverges, rendering an analog of Theorem \ref{t5.2} for 
 $T_{\alpha,\beta,\gamma,min}$ moot.

Apart from \cite{EZ78} in the special case $\gamma = (1-\beta)/(2+\alpha-\beta)$, implying $q \equiv 0$ and 
$\tau_{\alpha,\beta, (1-\beta)/(2+\alpha-\beta)}$ regular at $x=0$, we  did not find a treatment of the example 
in this section in the literature.  

\medskip 
\noindent 
{\bf Acknowledgments.} We are indebted to Man Kam Kwong, Lance Littlejohn, and Jonathan Partington for helpful discussions. 

\medskip

\noindent 
{\bf Data Availability Statement.} 
Data sharing is not applicable to this article as no datasets were generated or
analyzed during the current study.

\medskip

 
%
\end{document}